\documentclass[11pt, reqno]{amsart}
\usepackage{lmodern}
\usepackage{amsmath, amsthm, amssymb, amsfonts}
\usepackage[normalem]{ulem}
\usepackage{hyperref}
\usepackage[all,cmtip]{xy}
\usepackage{verbatim}
\usepackage{nccmath}
\usepackage{stmaryrd}
\usepackage{caption}
\usepackage{mathrsfs}
\usepackage{mathtools}
\usepackage{esvect}
\usepackage{cite}
\usepackage{bbm}
\usepackage{eucal}
\DeclarePairedDelimiter\ceil{\lceil}{\rceil}
\DeclarePairedDelimiter\floor{\lfloor}{\rfloor}
\usepackage{mathbbol}
\usepackage{tabularx}
\usepackage[toc,page]{appendix}

\usepackage{tikz-cd}

\theoremstyle{plain}
\newtheorem{thm}{Theorem}[section]
\newtheorem{cor}[thm]{Corollary}

\newtheorem{lemma}[thm]{Lemma}
\newtheorem{prop}[thm]{Proposition}

\renewcommand{\arraystretch}{2}

\theoremstyle{definition}
\newtheorem{defn}[thm]{Definition}
\newtheorem{example}[thm]{Example}
\makeatletter
\newcommand\ackname{Acknowledgements}
\if@titlepage
\newenvironment{acknowledgements}{%
	\titlepage
	\null\vfil
	\@beginparpenalty\@lowpenalty
	\begin{center}%
		\bfseries \ackname
		\@endparpenalty\@M
\end{center}}%
{\par\vfil\null\endtitlepage}
\else

\fi
\makeatother

\theoremstyle{remark}

\newcommand{\BA}{{\mathbb{A}}}

\newcommand{\BC}{{\mathbb{C}}}

\newcommand{\BL}{{\mathbb{L}}}
\newcommand{\BM}{{\mathbb{M}}}

\newcommand{\BQ}{{\mathbb{Q}}}
\newcommand{\BR}{{\mathbb{R}}}

\newcommand{\BT}{{\mathbb{T}}}

\newcommand{\BZ}{{\mathbb{Z}}}

\newcommand{\CB}{{\mathcal B}}

\newcommand{\CD}{{\mathcal D}}
\newcommand{\CE}{{\mathcal E}}
\newcommand{\CF}{{\mathcal F}}

\newcommand{\CH}{{\mathcal H}}
\newcommand{\CI}{{\mathcal I}}

\newcommand{\CL}{{\mathcal L}}
\newcommand{\CM}{{\mathcal M}}
\newcommand{\CN}{{\mathcal N}}
\newcommand{\CO}{{\mathcal O}}
\newcommand{\CP}{{\mathcal P}}

\newcommand{\CV}{{\mathcal V}}
\newcommand{\CW}{{\mathcal W}}

\newcommand{\CY}{{\mathcal Y}}
\newcommand{\CZ}{{\mathcal Z}}

\newcommand{\ch}{{\mathrm{ch}}}
\DeclareMathOperator{\Hilb}{Hilb}
\DeclareMathOperator{\FuM}{FM}

\newcommand{\tr}{{\mathrm{tr}}}
\newcommand{\CHom}{{\mathcal{H} om}}

\DeclareFontFamily{OT1}{rsfs}{}
\DeclareFontShape{OT1}{rsfs}{n}{it}{<-> rsfs10}{}
\DeclareMathAlphabet{\curly}{OT1}{rsfs}{n}{it}

\newcommand\Ext{\operatorname{Ext}}

\newcommand{\p}{\mathbb{P}}

\newcommand{\Mbar}{{\overline M}}

\newcommand{\GW}{\mathsf{GW}}

\newcommand{\Aut}{\mathrm{Aut}}

\begin{document}
	
	\title[Hilbert schemes  and Fulton--MacPherson compactifications]
	{Hilbert schemes of points and Fulton--MacPherson compactifications}

	\author{Denis Nesterov}
	\address{ETH Z\"urich, Departement Mathematik}
	\email{denis.nesterov@math.ethz.ch}
	\maketitle
	
\begin{abstract} We relate Hilbert schemes of points and Fulton--MacPherson compactifications by an interpolating stability condition. We then derive wall-crossings formulas and some applications for the enumerative geometry of Hilbert schemes. 
\end{abstract}
\setcounter{tocdepth}{1}
\tableofcontents
\section{Introduction} 

\subsection{Configuration space of points and its compactifications}	

Let $X$ be a smooth projective variety of dimension $d$ over the field of complex numbers $\BC$.  Consider an ordered configuration space of $n$ points on $X$, 
	\[ C_{[n]}(X):=\{(x_1, \dots, x_ n) \in X^n \mid x_i\neq x_ j, \text{ if } i\neq j\}.\]
	  It admits a free action of a symmetric group $S_n$,  the associated quotient is the unordered configuration space, 
	\[ C_n(X):=C_{[n]}(X)/S_n.\]
	Two prominent compactifications of $C_n(X)$ are:
	\begin{itemize}

	\item Hilbert schemes of points \cite{Grot}, which parametrise zero-dimensional subschemes of length $n$ on $X$, 
	\[ \Hilb_n(X)=\{ Z\subset X \mid \dim(Z)=0, \ell(Z)=n\}.\]
	\item 	 Fulton--MacPherson compactifications \cite{FM}, which parametrise $n$ distinct points on isotrivial semistable degenerations of $X$, also known as Fulton--MacPherson degenerations or simply bubblings of $X$, and denoted by $W$, 
	\begin{align*}
	 &\FuM_{[n]}(X)=\{(x_1, \dots, x_ n) \in W^n \mid \text{stability} \} /\hspace{-0.1cm} \sim,\\
	 	&\FuM_n(X)= [\FuM_{[n]}(X)/S_n],
	 	\end{align*}
 	such that points are identified by automorphisms of bubbles. 
		\end{itemize}

Hilbert schemes are home to numerous  mathematical marvels, like symmetric functions \cite{Haim}, infinite-dimensional Lie algebras \cite{Nak, Gro}, instantons \cite{ADHM}, etc. While smooth in dimensions one and two, they are very singular in higher dimensions. On the other hand, Fulton--MacPherson compactifications are  no less remarkable objects, generalising moduli spaces of stable marked curves with a fixed root curve $\Mbar_{C,n}$. Unlike Hilbert schemes, they are smooth in all dimensions, such that the complement of the configuration space $C_n(X)$ is a normal crossing divisor. 
\subsection{Interpolation} In this work,  we define a class of spaces, 
\[ \FuM^\epsilon_n(X), \quad \epsilon \in \BR_{>0}, \]
which interpolate between $\Hilb_n(X)$ and $\FuM_n(X)$, as we vary the stability parameter $\epsilon$. These spaces parametrise zero-dimensional subschemes on Fulton--MacPherson degenerations of $X$, 
\[ Z \subset W,\]
subject to some very simple conditions, the most important of which controls the length of $Z$ at its support and on the bubbles of $W$. Heuristically, as we vary $\epsilon$ in dimension two, it removes strata of the boundary divisor of $\Hilb_n(X)$, which also goes by the name of exceptional divisor, 
\[ \partial \Hilb_n(X)=\Hilb_n(X)\setminus C_n(X), \]
replacing them by components of the normal crossing boundary divisor\footnote{Note that $\FuM_n(X)$ is a stacky quotient of $\FuM_{[n]}(X)$, the action of $S_n$ has solvable stabilizers;  while $\partial \Hilb_n(X)$ is singular,   $\partial \FuM_n(X)$ is a stacky normal crossing divisor. }  of $\FuM_n(X)$. 
Similarly, as we vary $\epsilon$ in higher dimensions, it  removes all those monstrous irreducible components of $\Hilb_n(X)$, leaving a smoothed version of the main component - $\FuM_n(X)$. 

\subsection{Tautological integrals}
However, the classical geometric aspects of this interpolation is not the focus of this work. Instead, we study its implications for the enumerative geometries of the two spaces. Assuming $d\leq 3$, we derive a wall-crossing formula, Corollary \ref{maincor}, which relates tautological integrals,

\[\int_{[\Hilb_{n}(X)]^{\mathrm{vir}} }\prod_i \tau_{k_i}\quad \text{and} \quad \int_{[\FuM_{[n]}(X)] }\prod_j \psi_j^{k_j}.\]
Note that the appearance of  ordered Fulton--MacPherson spaces is deliberate and not significant for our purposes, since the tautological integrals on $\FuM_{[n]}(X)$  and $\FuM_{n}(X)$ are related by a factor of $n!$. 

The wall-crossing formula is not a statement of equivalence of the integrals above, nor it prioritises any of the two. It rather asserts universality of the third type of integrals which appears in the formula, referred to as $I$-functions.  By Proposition \ref{univer}, in our context, $I$-functions are essentially the equivariant tautological integrals on $\Hilb_n(\BC^d)$, 
 \[\int_{[\Hilb_{n}(\BC^d)]^{\mathrm{vir}} }\prod_i \tau_{k_i} \in \BQ[a_i^\pm].\]
 The  formula  expresses integrals on $\Hilb_{n}(X)$ in terms of those on  $\Hilb_n(\BC^d)$ and $\FuM_{[n]}(X)$. In some sense, $\FuM_{[n]}(X)$ is responsible for the correction in the comparison of $\Hilb_n(X)$ with $\Hilb_n(\BC^d)$.  By \cite{FM},  the cohomology ring of $\FuM_{[n]}(X)$ has a universal presentation whose relations involve Chern classes of $X$ and boundary divisors of $\FuM_{[n]}(X)$.  The immediate consequence of the wall-crossing formula is therefore  the universality of  integrals on $\Hilb_n(X)$, Corollary \ref{unicor}, which was considered for surfaces in \cite{EGL}. Our methods provide a different perspective on this phenomenon, extending it to equivariant integrals in dimension three. 
 
 We also illustrate how the wall-crossing formula works by computing the simplest kind of tautological integral of $\Hilb_n(X)$ - the topological Euler characteristics $e(\Hilb_n(X))$ - providing another proof for the threefold calculations of \cite{LiDT,LP}. The computation showcases the significance of $\psi$-classes on $\FuM_{[n]}(X)$, to which we dedicate a whole section, Section \ref{comp}.


\subsection{Affine spaces} The wall-crossing can also be applied to $\Hilb_n(\BC^d)$ themselves. However, for it to give something meaningful, there should be poles in equivariant variables, as is observed in Section \ref{secaff}. Using the wall-crossing, we provide an effective way to determine tautological integrals equivariant with respect to the diagonal torus on $\Hilb_n(\BC^2)$ for all $n$ in terms of integrals for finitely many values of $n$, Corollary \ref{strangecor}. 

The spaces $\FuM_n(\BC^d)$ are highly understudied. In light of our interpolating stability, it is natural to expect that their geometry might have equally interesting representation-theoretic and combinational facets like  $\Hilb_n(\BC^d)$. For example, graphs and trees play essentially the same role for them as partitions for $\Hilb_n(\BC^d)$ - they enumerate various classes on $\FuM_n(\BC^d)$, as well as the torus fixed components, see Section \ref{sector}. Moreover, the complex Fulton--MacPherson operads \cite{Getz} give rise to similar recursive structures on the cohomology of $\FuM_n(\BC^d)$ as the action of the Heisenberg algebra \cite{Nak,Gro} in the case of $\Hilb_n(\BC^2)$. On the other hand, the cohomology of $\FuM^\epsilon_n(\BC^2)$ will contain both  Nakajima-type classes as well as Fulton--MacPherson boundary classes. 

\subsection{Relations to other work} Our stability condition and the associated wall-crossing formula is a member of the vast family of related stabilities.  Among them, the theory of GIT quasimaps \cite{CFKM, CFK14, CFKmi} had the most significant impact on author's understanding of this phenomenon. Consequently, many of our techniques are also imported from there, the most profound of which is Zhou's theory of entangled tails \cite{YZ}, which allows to prove the wall-crossing formulas. We refer to  \cite[Section 1.4]{NuG} for a non-exhaustive list of situations in which such stabilities arise. 

The closest relative of the stability condition considered in this work is Hassett's weighted stable curves \cite{Has}. In fact, in dimension one, our stability essentially specialises to Hassett's. In higher dimensions, there exists a more direct generalisation  of  Hassett's stability, considered in \cite{Rou}. We state a wall-crossing formula for this stability too in Section \ref{secnfold}. It provides a useful source of constraints on integrals of $\psi$-classes over $\FuM_{[n]}(X)$, which we use for computations in Section \ref{sectioncomp}. 

Perhaps the  result that is closest in spirit to our wall-crossing formula is seemingly unrelated Gromov--Witten/Hurwitz correspondence of Okounkov--Pandharipande \cite{OPcom}. There are good reasons for this. Firstly, it can be proved by exactly the same techniques, as is shown in \cite{NSc}. Secondly, it has the same structure, even though it might be not so evident from the angle of Okounkov--Pandharipande's formulation. More precisely, it expresses Gromov--Witten invariants of a curve $C$ in terms of Hurwitz invariants of $C$ and Hodge integrals, which are completed cycles in disguise.  Hodge integrals can in turn be viewed as equivariant Gromov--Witten invariants of the affine line $\BC$, which is exactly the $I$-function of the Gromov--Witten/Hurwitz correspondence. The correspondence can therefore be interpreted as

\[ \text{GW($C$)}=\text{H($C$})+ \GW(\BC). \]
 On the other hand, our wall-crossing formula reads
 \[\text{Hilb}(X)=\text{FM}(X)+\Hilb(\BC^d). \]
 
Of course, the major difference is that the relevant Hodge integrals can be given a very explicit  form  in terms of completed cycles. This is arguably harder for Hilbert schemes. Certain $K$-theoretic variants of the tautological integrals were computed in \cite{Arb, Arb2, WZ} for $\Hilb_n(\BC^2)$. 
\subsection{Acknowledgements} I thank Donggun Lee and Noah Arbesfeld for useful discussions on related topics, and Rahul Pandharipande for drawing my attention to \cite{CGK}. The help of Yannik Schuler with computer calculations is greatly appreciated. 

Parts of this work were written during my stay in KIAS, Korea. I am grateful to KIAS for their hospitality  and  Hyeonjun Park for the invitation. Moreover, it is impossible to overestimate the role that the ideas of the prominent member of KIAS, Bumsig Kim, and his collaborators played in this work. Not only he and his collaborators laid down the foundations of quasimaps and their wall-crossing, but he was also among the first ones to study Fulton--MacPherson spaces in the context of  enumerative geometry. In many ways, this work is an upside-down panorama of Kim's discoveries.
 
  This research was supported by ERC grant “Refined invariants in combinatorics, low-dimensional topology and
  geometry of moduli spaces” No. 101001159 at University of Vienna, and  by a Hermann-Weyl-instructorship from the
  Forschungsinstitut f\"ur Mathematik at ETH Z\"urich.

\section{Moduli spaces of $\epsilon$-weighted subschemes}
\subsection{Preliminaries} Throughout the article, $X$ is a smooth quasi-projective  variety over the field of complex numbers $\BC$. We allow $X$ to carry an action of a torus $T=(\BC^*)^k$, requiring that $X^T$ is projective. If $X$ does not carry an action of $T$, then $X^T:=X$.  We will denote its $T$-equivariant cohomology as follows,
\[ H^*(X):=H^*_T(X,\BQ). \]
For the most part, we will assume that $X$ is  non-equivariant and make remarks about the equivariant case when necessary. 

Various spaces of ordered and unordered points will feature in this article. We reserve the notation $[n]$ for the set $\{1,\dots, n\}$ and use it whenever ordered points are considered. On the other hand, $n$ denotes an integer and is used in the case of unordered points. 
	\subsection{Fulton--MacPherson degenerations} \label{FMdeg} We start with the recollection of Fulton--MacPherson compactifications and degenerations, the main references for which are \cite{FM} and \cite{KKO}.  
	Consider the Fulton--MacPherson compactification of configuration spaces of points $\FuM_{[n]}(X)$, constructed in \cite{FM}, as an inductive blow-up. 
	 Let 
	 \[p_n \colon \CW \rightarrow \FuM_{[n]}(X) \quad \text{and} \quad q_{n } \colon \CW \rightarrow  X\]
	 be  the universal family\footnote{Note that $\FuM_{[n]}(X)$ is not constructed as a moduli space, hence, to be precise, $\CW$ is not a universal family. Nevertheless, we will refer to it as a universal family, slightly abusing the terminology.} of  $\FuM_{[n]}(X)$ and the associated universal contraction, also constructed in \textit{loc.\ cit.}
	 
	 \begin{defn}A pair consisting of a variety and a map, 
	 \[ (W, q \colon W \rightarrow X),\] is a Fulton--MacPherson (FM) degeneration  of $X$, if it is isomorphic to a  fiber of $p_{n}$ over a closed point in $\FuM_{[n]}(X)$ for some $n$. An end component of $W$ is a component whose valency in the intersection graph\footnote{The graph whose vertices are irreducible components of a variety, such that there is an edge between two vertices, if the corresponding components intersect. } is equal to $1$, except if it is a blow-up of $X$; it is isomorphic to $\p^d$. A ruled component is a component of valency 2, except if it is a blow-up of $X$; it is isomorphic to $\mathrm{Bl}_0(\p^d)$.  
	 \end{defn}
 
  Alternatively, a FM degeneration
	 $W$ can be constructed as a central fiber of an iterated  blow-up of $X\times \BC$, such that the centres of blow-ups are points in the regular locus of the central fiber of the previous blow-up. An example of a FM degeneration is
	 \[ W=\mathrm{Bl}_p(X) \sqcup_{E=D}\p^d,\]
	 where we attach two components at the exceptional divisor $E\subset  \mathrm{Bl}_p(X)$  and a divisor at infinity $D \subset \p^d$. If $d=1$, $W$ is a curve with $\p^1$-bubbles. Very often, we will refer to components of $W$, which are not blow-ups of $X$, as bubbles. 
	   Automorphisms of a FM degeneration are automorphisms of $(W,q)$ which fix  $q \colon W \rightarrow X$. We will denote the group of automorphisms of $(W,q)$ simply by $\Aut(W)$.
	\\ 
	
	An algebraic stack of FM degenerations  is constructed as a groupoid in \cite[Section 2.8]{KKO},
	\[ \CF\CM(X):=[\FuM_{[n,n]}(X) \rightrightarrows \FuM_{[n]}(X)],\]
	where $\FuM_{[n,n]}(X)$ is the FM compacitification with two sets of markings, while maps are given by forgetting any of these two sets of markings. This groupoid identifies stable marked FM degenerations which have different markings but the same underlying FM degeneration. Strictly speaking, by its definition,  $\CF\CM(X)$ parametrises FM degenerations with at most $n$ irreducible components. Since all of our moduli spaces are bounded, we will always assume that $n$ in the definition of $\CF\CM(X)$ is just large enough. 
	
	Let $\CF\CM_{[m]}(X)$ be a stack of FM degenerations with ordered $m$ markings, $(W,\underline{p}):=(W,p_1,\dots, p_m)$. The space $\FuM_{[m]}(X)$ is contained in  $\CF\CM_{[m]}(X)$,
	\[ \FuM_{[m]}(X) \subset \CF\CM_{[m]}(X),\] 
	as on open subspace of stable marked FM degenerations, i.e.\ $(W,\underline{p}) \in \FuM_{[m]}(X)$, if 
	\begin{enumerate}
		\item $\underline{p}$ is contained in the smooth locus of $W$, 
		\item $\Aut(W,\underline{p})$ is finite\footnote{This implies that $\Aut(W,\underline{p})$ is in fact trivial. },
			\item $p_i\neq p_j$, if $i\neq j$.
		\end{enumerate}
	
	Since $\FuM_{[n]}(X)$ was constructed via blow-ups, a clarification is needed here: by imposing the above stability, we obtain an open algebraic subspace $\FuM_{[m]}(X)'$ of $\CF\CM_{[m]}(X)$. By the definition of $\CF\CM_{[m]}(X)$, we have a map $\FuM_{[m]}(X)\rightarrow \FuM_{[m]}(X)'$, which is clearly bijective on geometric points. Since both spaces are smooth, by Zariski's main theorem, we obtain that the map must be an isomorphism. 
	 
	 \subsection{Stability}
	 In what follows, $W^{\mathrm{sm}}$ denotes the smooth locus of a FM degeneration $W$, $\ell_x(Z)$ is the length of a zero-dimensional subscheme $Z$ at a point $x \in W$, and $\ell(Z)$ is the total length of the subscheme. 
	 
	 \begin{defn}
	 	Given $\epsilon \in \BR_{>0}$. Let $W$ be a FM degeneration together with a zero-dimensional subscheme $Z\subset W$. The pair $(W,Z)$  is said to be $\epsilon$-weighted, if 
	 	\begin{enumerate}
	 		\item  $Z \subset W^{\mathrm{sm}}$, 
	 		\item  for all $x \in W$, $\ell_x(Z)\leq 1/\epsilon$,
	 		\item  for all end components $\p^{d} \subset W$, $\ell(Z_{|\p^{d}})>1/\epsilon$, 
	 		\item  the group $\{g \in \Aut(W) \mid g^*Z=Z\}$ is finite.
 		\end{enumerate}	
 	 \end{defn}
 	
 	\subsection{Moduli spaces}
 	
 	For each $\epsilon \in \BR_{>0}$, we define the space of $\epsilon$-weighted pairs $(W,Z)$,
 	\[ \FuM^{\epsilon}_n(X):=\{\epsilon\text{-weighted} \ (W,Z) \mid \ell(Z)=n\} / \sim,\] where $``\sim"$ denotes that subschemes are identified by automorphisms of FM degenerations. 
 	 More precisely,  the space $\FuM^{\epsilon}_n(X)$ is a substack of the relative Hilbert scheme of zero-dimensional subschemes of the universal family $\CW \rightarrow \CF\CM(X)$, 
 	\[  \FuM^{\epsilon}_n(X) \subseteq \Hilb_n(\CW/\CF\CM(X)). \]
 	In particular, a family of pairs in $\FuM^{\epsilon}_n(X)$ is defined via the family of objects in $\Hilb_n(\CW/\CF\CM(X))$. 
 	It is not difficult to see that $\epsilon$-weightedness is an open condition, because the length of a point and the total length on an end component can only increase and decrease under specialisations, respectively. Hence $\FuM^{\epsilon}_n(X)$ is in fact an open substack. In particular, it is algebraic and quasi-separated, since $\Hilb_n(\CW/\CF\CM(X))$ is so by \cite[\href{https://stacks.math.columbia.edu/tag/0DLX}{Section 0DLX}]{stacks-project} and the representability of $\CW \rightarrow \CF\CM(X)$.  Moreover, by the condition (4) of $\epsilon$-weightedness, it is Deligne--Mumford and of finite type.
 	\\
 	 	
 	 If $\epsilon \leq 1/n$, the stability prohibits all end components, hence $\FuM^\epsilon_n(X)$ parametrises zero-dimensional subschemes of length $n$ on $X$, 
 	\[ \FuM^\epsilon_n(X)=\Hilb_n(X), \quad \text{if } \epsilon \leq 1/n.\] 
 	If $\epsilon=1$, all subschemes are of length 1 at all points of their support, hence $\FuM^\epsilon_n(X)$  parametrises $n$-tuples of distinct unordered points on FM degenerations of $X$,  
 	\[ \FuM^\epsilon_n(X)=\FuM_n(X), \quad \text{if } \epsilon=1.\]
 	If $\epsilon>1$, then $\FuM^\epsilon_n(X)$ is empty, 
 	\[ \FuM^\epsilon_n(X)=\emptyset , \quad \text{if } \epsilon >1.\]
 	For values of $\epsilon$ between the extremal ones, a pair $(W,Z)$ in $\FuM^\epsilon_n(X)$ might have both end components at $W$ and non-reducedness at $Z$. 
 	 \subsection{Marked $\epsilon$-weighted subschemes}
 	 We will need a more general version of $\FuM^\epsilon_n(X)$ which parametrises $\epsilon$-weighted $0$-subschemes together with marked points which are not affected by the $\epsilon$-weightedness.  
 	\begin{defn} \label{defnweight}
 		Given $\epsilon \in \BR_{>0}$.  Let $W$ be a FM degeneration together with a zero-dimensional subscheme $Z\subset W$ and $m$ distinct ordered point $\underline{p}:=\{p_1,\dots,p_m\}$. The triple $(W,Z,\underline{p})$ is $\epsilon$-weighted, if 
 		\begin{enumerate} 
 				\item $\underline{p} \cap Z=\emptyset$,
 			\item $\underline{p} \subset W^{\mathrm{sm}}$ and $Z\subset W^{\mathrm{sm}}$, 
 			\item for all $x \in W$, $\ell_x(Z)\leq 1/\epsilon$,
 			\item  for all end components  $\p^{d} \subset W$, such that $\underline{p} \cap \p^{d}=\emptyset$,   $\ell(Z_{|\p^{d}})>1/\epsilon$, 
 			\item the group $\{g \in \Aut(W, \underline{p}) \mid g^*Z=Z\}$ is finite.
 		\end{enumerate} 	
 	\end{defn}
 	
 	For each $\epsilon \in \BR_{>0}$, we define the space of $\epsilon$-weighted marked subschemes
 	\[ \FuM^{\epsilon}_{n,[m]}(X):=\{\epsilon\text{-weighted } (W,Z,\underline{p}) \mid  \ell(Z)=n, |\underline{p}|=m\} / \sim.\] 
 The space $\FuM^{\epsilon}_{n,[m]}(X)$ is a substack of the relative Hilbert scheme of zero-dimensional subschemes of $\CW \rightarrow \CF\CM_{[m]}(X)$, 
 	\[  \FuM^{\epsilon}_{n,[m]}(X) \subseteq \Hilb_n(\CW/\CF\CM_{[m]}(X)).\]
  As before, for the same reasons,  $\FuM^{\epsilon}_{n,[m]}(X)$ is a quasi-separated Deligne--Mumford stack of finite type.  
 	These spaces clearly specialise to  $\FuM^{\epsilon}_{n}(X)$ by setting $m=0$, and to $\FuM_{[m]}(X)$ by setting $n=0$. In fact, they will play a central role in the wall-crossing formulas. Even if we want to compare just  $\FuM_{n}(X)$ and $\Hilb_n(X)$, spaces $\FuM^{\epsilon}_{n,[m]}(X)$ will be necessary. 
 	
 For all $\epsilon$, $n$ and $m$, spaces $\FuM^{\epsilon}_{n,[m]}(X)$ contain partially ordered configuration spaces of points, 
 \[C_{[n+m]}(X)/S_{n}, \]
 as open subschemes. Note that the points $\underline{p}$ are ordered, while supports of subschemes $Z$ are not. For this reason, we quotient out the order only on first $n$ points in the configuration space above. 
 
 	If a torus $T$ acts on $X$, then it also acts on $\FuM_{[n]}(X)$, such that the action descends to $\CF \CM_{[m]}(X)$. It therefore also induces an action on $\FuM^{\epsilon}_{n,[m]}(X)$.

 	
 	\subsection{Properness}
 		
 		\begin{prop} 
 		If $X^T$ is projective, then the torus-fixed locus $\FuM^{\epsilon}_{n,[m]}(X)^T$ is proper. 
 		\end{prop}
 	
 		\begin{proof} We will assume that $X$ is projective and non-equivariant. Since $\FuM^{\epsilon}_{n,[m]}(X)^T$ consists of points supported over $X^T$,  the arguments easily extend to the non-projective case. 
 			
 			We use the valuative criteria of properness for quasi-separated Deligne--Mumford stacks of finite type. Let $\Delta$ be the spectrum of a discrete valuation ring, let $\Delta^\circ \in \Delta$ and $0\in \Delta$ be its generic and closed points, respectively. Let $(\CW^\circ,\CZ^\circ,\underline{p}^\circ)$ be a family of $\epsilon$-weighted triples over $\Delta^\circ$. By taking the associated reduced subscheme of $\CZ^\circ$, we obtain $k$ distinct $\Delta^\circ$-points of $\CW^\circ$,
 			\[\CZ^\mathrm{red,\circ}=\underline{x}^\circ=\{x^\circ_1, \dots, x^\circ_k\} \subset \CW^\circ,\]
 			which we order and  to which we attach multiplicities 
 			\[\{n_1, \dots, n_k\} \subset (\BZ_{>0})^k,\] 
 			defined as $n_i:=\ell_{x^\circ_i}(\CZ^\circ)$. By the $\epsilon$-weightedness of $\CZ^\circ$, we have 
 			\begin{equation*} \label{eqweight}
 			n_i\leq 1/\epsilon. 
 			\end{equation*}

 			Since $\CZ^\circ \cap \underline{p}^\circ=\emptyset$, the triple $(\CW^\circ,\underline{x}^\circ,\underline{p}^\circ)$ defines a $\Delta^\circ$-point of $\FuM_{[k+m]}(X)$. By the properness of $\FuM_{[k+m]}(X)$, there exists an extension $(\CW, \underline{x},\underline{p})$ of $(\CW^\circ,\underline{x}^\circ,\underline{p}^\circ)$ over $\Delta$, such that the central fiber $(\CW_0, \underline{x}_0,\underline{p}_0)$ is a stable marked FM degeneration. Let $\p^{d} \subset \CW_0$ be an end component, such that $  \underline{p}_0 \cap  \p^{d} =\emptyset$, and let $(x_{i_1,0}, \dots, x_{i_\ell,0})\subset \underline{x}_0$ be the subset of $\underline{x}_0$ contained in $\p^{d}$. If  
 			\begin{equation} \label{eq1}
 				\sum^{\ell}_{j=1} n_{i_j}>1/\epsilon,
 			\end{equation}
 		we leave $\p^d$ intact. Otherwise, we contract it, which is possible, because in this case $\p^d$ cannot be a limit of an end component in the generic fiber by the $\epsilon$-weightedness of the generic fiber, we refer to \cite[Section 2.2]{NuG} for contractions and blow-ups of FM degenerations.  We  thereby obtain a stable FM degeneration, such that markings $\underline{x}$ might overlap. We then repeat the procedure for all end components of this new FM degeneration and of its further contractions. In this way, we obtain a stable FM degeneration $(W', \underline{x},\underline{p})$ with possibly overlapping markings $\underline{x}$, such that all end components $\p^{d} \subset \CW'_0$ satisfy (\ref{eq1}), while the total multiplicity of possibly overlapping markings at each point of $W'_0$ satisfy
 		\begin{equation} \label{ineq}
 		\sum n_{i_j}\leq 1/\epsilon
 		\end{equation}
 		by the choice of the contracted components and the $\epsilon$-weightedness of the generic fiber.
 		Since $\CW' \rightarrow \Delta$ is projective,  the relative Hilbert scheme of zero-dimensional subschemes $\Hilb_n(\CW'/\Delta)$ is proper. Hence there exists an extension $\CZ \subset \CW'$ of $\CZ^\circ \subset \CW^\circ$ over  $\Delta$.  By the construction of $\CW'$ and the deformation invariance of the total length of subschemes, the central fiber $(\CW_0,\CZ_0,\underline{p}_0)$ is $\epsilon$-weighted. This shows the existence part. 
 		
 		The uniqueness part follows from the uniqueness of the FM degeneration $(W',\underline{x},\underline{p})$ and the properness of $\Hilb_n(\CW'/\Delta)$. Indeed, $(W',\underline{x},\underline{p})$ is uniquely characterised by stability and  the inequality $(\ref{ineq})$, since another triple  $(W'',\underline{x},\underline{p})$ would be related to $(W',\underline{x},\underline{p})$ by  blow-ups and blow-downs, which either introduce unstable components or violate the inequality (\ref{ineq}). 
 		\end{proof}
 		
 \subsection{Obstruction theory}
  
 	Since   $\FuM^{\epsilon}_{n,[m]}(X)$ is an open substack of the relative Hilbert scheme $\Hilb_n(\CW/\CF\CM_{[m]}(X))$, there exists a natural obstruction theory relative to $\CF\CM_{[m]}(X)$ given by the obstruction theory of ideal sheaves. Let  
 	\begin{align*}
 	 &\CI \rightarrow \CW\times_{\CF\CM_{[m]}(X)} \FuM^{\epsilon}_{n,[m]}(X), \\
 	 & p_{\CW} \colon \CW\times_{\CF\CM_{[m]}(X)}\FuM^{\epsilon}_{n,[m]}(X)  \rightarrow  \FuM^{\epsilon}_{n,[m]}(X), 
 	 \end{align*}
 	be the universal ideal sheaf associated to the universal subscheme $\CZ$ and the canonical projection, respectively. Since for every FM degeneration $(W,q \colon W \rightarrow X)$, we have
 	\[ Rq_*\CO_W=\CO_X,\]
 	there is a natural identification, 
 	\begin{equation}\label{trace}
 	Rp_{\CW*}(\CO_{\CW\times_{\CF\CM_{[m]}(X)}\FuM^{\epsilon}_{n,[m]}(X)})= H^*(\CO_X)\otimes \CO_{\FuM^\epsilon_{n,[m]}(X)}. 
 	\end{equation}
 	Consider the following complex, 
 	
 	\begin{align*}
 	& \BT^{\mathrm{vir}}_{\FuM_{n,[m]}^\epsilon(X)/\CF\CM_{[m]}(X)}:=R\CHom_{p_\CW}(\CI,\CI)_0=\mathrm{cone}(\tr), \\
 	& \tr\colon R\CHom_{p_\CW}(\CI,\CI) \rightarrow H^*(\CO_X)\otimes \CO_{\FuM^\epsilon_{n,[m]}(X)},
 	\end{align*}
 where $\mathrm{tr}$ is the trace morphism for which we used the identification (\ref{trace}). 
 	
 	\begin{prop} The complex $\BT^{\mathrm{vir}}_{\FuM_{n,[m]}^\epsilon(X)/\CF\CM_{[m]}(X)}$ defines a relative obstruction theory, 
 		\[\phi\colon \BT_{\FuM_{n,[m]}^\epsilon(X)/\CF\CM_{[m]}(X)} \rightarrow \BT^{\mathrm{vir}}_{\FuM_{n,[m]}^\epsilon(X)/\CF\CM_{[m]}(X)},  \]
 		which is perfect, if $d\leq 3$. 
 	\end{prop}
 
 \begin{proof} It is a relative version of \cite{Th}, see also \cite{HuT}. A more modern approach is provided by the derived algebraic geometry \cite{TV2,Lur}.  
 	\end{proof}
 
 Since $\CF\CM_{[m]}(X)$ is smooth, we obtain an absolute obstruction theory, defined by the following triangle, 
 \begin{equation} \label{cotangent}
 \BT^{\mathrm{vir}}_{\FuM_{n,[m]}^\epsilon(X)/\CF\CM_{[m]}(X)}  \rightarrow \BT^{\mathrm{vir}}_{\FuM_{n,[m]}^\epsilon(X)} \rightarrow \BT_{\CF\CM_{[m]}(X)} \rightarrow.
 \end{equation}
 This furnishes $\FuM_{n,[m]}^\epsilon(X)$ with a virtual fundamental class \cite{BF}, if $d\leq 3$, 
 \[ [\FuM_{n,[m]}^\epsilon(X)]^{\mathrm{vir}} \in H_*(\FuM_{n,[m]}^\epsilon(X)).\]
 By the results in the next section, this class coincides with the standard fundamental class, if $d\leq 2$. 

  \subsection{Smoothness and connectedness}
 
 \begin{prop}
  If $d\leq2$, then $\FuM_{n,[m]}^\epsilon(X)$ is smooth and connected of dimension $d(n+m)$, and 
   \[\BT^{\mathrm{vir}}_{\FuM_{n,[m]}^\epsilon(X)}\cong T_{\FuM_{n,[m]}^\epsilon(X)}. \]
 \end{prop}

\begin{proof} 

We will establish the claims of the proposition in the reverse order. Firstly,  traceless $\Ext^2$ groups vanish for ideal sheaves of zero-dimensional subschemes on smooth surfaces and curves, and  $\BT^{\mathrm{vir}}_{\FuM_{n,[m]}^\epsilon(X)/\CF\CM_{[m]}(X)}$ is an obstruction-theory complex, thus we have 
\[
\BT^{\mathrm{vir}}_{\FuM_{n,[m]}^\epsilon(X)/\CF\CM_{[m]}(X)}
\cong T_{\FuM_{n,[m]}^\epsilon(X)/\CF\CM_{[m]}(X)},
\]
which by the triangle (\ref{cotangent}) and smoothness of $\CF\CM_{[m]}(X)$ then implies that $\BT^{\mathrm{vir}}_{\FuM_{n,[m]}^\epsilon(X)}\cong T_{\FuM_{n,[m]}^\epsilon(X)}$. Moreover, again by the vanishing of traceless $\Ext^2$ and smoothness of $\CF\CM_{[m]}(X)$,  we conclude that the Zariski tangent spaces of $\FuM_{n,[m]}^\epsilon(X)$ are of constant dimension equal to the dimension of the configuration space $C_{[n+m]}(X)/S_{n}$, i.e.\  $d(n+m)$. 

Consider now the closure of the configuration space $C_{[n+m]}(X)/S_{n}$ inside $\FuM_{n,[m]}^\epsilon(X)$, and denote it by $\CM$. If there was an another irreducible component of $\FuM_{n,[m]}^\epsilon(X)$  intersecting $\CM$, then the dimension of the Zariski tangent space of a point of intersection would be greater than  $d(n+m)$ contradicting the conclusion above. Hence all other irreducible components of $\FuM_{n,[m]}^\epsilon(X)$  must be disjoint from $\CM$. Moreover, $\CM$ must be smooth, since the dimension of its Zariski tangent spaces  are equal to its dimension. 


It remains to prove that $\FuM_{n,[m]}^\epsilon(X)$  is connected.  To see this, it is enough to show that for all $\epsilon$-weighted triples $(W,Z,\underline{p}) \in  \FuM_{n,[m]}^\epsilon(X)$, there exists a family over a connected base  which connects $(W,Z,\underline{p})$ with a point in $C_{[n+m]}(X)/S_{n} \subset \FuM_{n,[m]}^\epsilon(X)$.  This is achieved as follows. By the connectedness of Hilbert schemes on a smooth surface or a curve established in \cite{Fog}, we can deform any non-reduced subschemes on any end component $\p^d$ of $W$ to a  subscheme which is reduced on $\p^d$ (and still $\epsilon$-weighted). By the connectedness of Fulton--MacPherson spaces, we then can smooth out the chosen end component, such that the resulting triple $(W,Z,\underline{p})$ is also  $\epsilon$-weighted. We apply this procedure inductively to all end components, and then to all non-reduced subschemes supported on $X$. The resulting pair will be contained in $C_{[n+m]}(X)/S_{n} $. We conclude that  $\FuM_{n,[m]}^\epsilon(X)$ is connected, since $C_{[n+m]}(X)/S_{n} $ is connected. Hence  $\FuM_{n,[m]}^\epsilon(X)=\CM$, and therefore it is smooth of dimension $d(n+m)$. 
\end{proof}

A consequence of the proposition above is that   $\FuM_{n,[m]}^\epsilon(X)$ provides distinct\footnote{The fact that they are distinct requires an argument; this can be seen by contemplating what happens with the exceptional divisor of $\Hilb_n(X)$, when we vary $\epsilon$. If $n=2$ and $m=0$, they agree up to stabilizers. } smooth and connected compactifications of the configuration spaces of points $C_{[n+m]}(X)/S_{n}$, if $d\leq 2$. In higher dimensions, the geometry of $\Hilb_n(X)$ is notoriously complicated - one can essentially find any type of singularity inside them \cite{Jel}.  It is also true for $\FuM_{n,[m]}^\epsilon(X)$, as it inherits this complexity from  $\Hilb_n(X)$,  unless $n=0$ or $\epsilon =1$ (or if $\epsilon$ is very close to $1$). Hence moduli spaces $\FuM_{n}^\epsilon(X)$ interpolate between  highly singular spaces $\Hilb_n(X)$ and smooth orbifold spaces $\FuM_n(X)$ whose boundary is a normal crossing divisor.

If $d=3$, using Serre's duality, we can calculate the virtual dimension of spaces $\FuM_{n,[m]}^\epsilon(X)$, 
\[ \mathrm{vir.\dim}(\FuM_{n,[m]}^\epsilon(X))=m, \]
in particular, if $m=0$, then $\mathrm{vir.\dim}(\FuM_{n}^\epsilon(X))=0$.

\section{Tautological integrals}

\subsection{Descendents of Hilbert schemes} 
From now on, we assume that $d\leq 3$. We start with tautological classes of Hilbert-schemes type. Let 	
\[ \CI \rightarrow \CW \times_{\CF\CM_{[m]}(X)}  \FuM_{n,[m]}^\epsilon(X) \] 
be the universal ideal. By the definition of FM degenerations, we have a universal contraction to $X$, 
\[q_\CW \colon \CW \times_{\CF\CM_{[m]}(X)}  \FuM_{n,[m]}^\epsilon(X) \rightarrow X  \]
 For a class $\gamma \in H^*(X)$, we define 
\[ \tau_k(\gamma):=p_{\CW*}(\ch_{k}(\CI)\cdot q^*_\CW\gamma) \in H^*(\FuM_{n,[m]}^\epsilon(X)), \]
observe that for $\epsilon \leq 1/n$ and $m=0$, these are standard tautological classes on Hilbert schemes of points, usually referred to as descendents.  

\subsection{Descendents of Fulton--MacPherson spaces} \label{secdescFM}

Let us now consider the tautological classes of Fulton--MacPherson type. Firstly, we have another kind of projections given by composing the universal contraction $q_{\CW}$ and a universal section $s_j \colon \FuM^\epsilon_{n,[m]}(X) \rightarrow  \CW \times_{\CF\CM_{[m]}(X)}  \FuM_{n,[m]}^\epsilon(X)$, 

\begin{align*}
	\pi_j =q_{\CW} \circ s_j \colon \FuM_{n,[m]}^\epsilon(X)  \rightarrow X, \quad j=1,\dots, m.
\end{align*}
More importantly, we have $\psi$-classes, 
\[ \BL_{j}=s^*_jT_{p_\CW}^\vee,\quad  \psi_j:= (\mathrm{c}_1(\BL_j)-\pi_j^*\mathrm{c}_1(X))/d, \quad j=1,\dots,m. \]
We refer to Section \ref{comp} for more details about $\psi$-classes in our context. There are other ways of defining them, e.g.\ they can be expressed in terms boundary divisors,  
\[\psi_j= \sum_{j \in J} D_J. \]

There exists another natural family of classes, 
\[ \tilde{\psi}_j:=\mathrm{c}_{d}(\BL_j)=\sum^{d}_{\ell=0}(-1)^\ell \psi_j^{d-\ell}  \pi_j^*\mathrm{c}_\ell(X),\]
where the second equality holds by Lemma \ref{Li}. In particular,  if $d=1$,  classes $\tilde{\psi}_j$ are more conventional $\psi$-classes and they differ from our classes $\psi_j$ by a twist $-\pi_j^*\mathrm{c}_1(X)$. We deviate from the standard notation, because our classes $\psi_j$ play a more important role in this work by appearing in the wall-crossing formula, Theorem \ref{maintheorem}. On the other hand, classes $\tilde{\psi}_j$ satisfy a generalisation of dilaton and string equations, Section \ref{sectiondil}. Classes of the form $\psi_j^{k_j}\pi^*_j\gamma_j$ are also referred to as descendents.

\begin{defn} \label{integrals1}
Assume $d\leq 3$. Let $\gamma_i,\gamma_j' \in H^*(X)$ be some classes, \textit{tautological integrals} on $\FuM_{n,[m]}^\epsilon(X)$ are defined as follows, 
	\begin{multline*}
		 \langle \tau_{k_1}(\gamma_1)\dots \tau_{k_N}(\gamma_N) \mid \psi_1^{k'_1}\gamma'_1\dots \psi_m^{k'_m}\gamma'_m\rangle_n^{\epsilon}:=\\ 
		 \int_{[\FuM_{n,[m]}^\epsilon(X)]^{\mathrm{vir}} }\prod_i \tau_{k_i}(\gamma_i) \prod_{j} \psi^{k'_j}_j\pi^*_j\gamma'_j.
	 \end{multline*}  
 If $X$ is not projective, we use the (virtual) localisation \cite{GP} applied to $\FuM_{n,[m]}^\epsilon(X)^T$ to define the integrals above. 
 
  Note that the Hilbert-scheme insertions are not restricted by $n$, while the number of Fulton--MacPherson insertions is equal to $m$. Very often we will use the following shortened notation, 
\begin{equation*} \langle \prod^{}_{i} \tau_{k_i}(\gamma_i) \mid \prod^{}_{j} \psi^{k'_j}\gamma'_j \rangle^\epsilon_n= \langle \tau_{k_1}(\gamma_1)\dots \tau_{k_N}(\gamma_N) \mid \psi^{k'_1}\gamma'_1\dots \psi^{k'_m}\gamma'_m\rangle_n^{\epsilon},
	\end{equation*}
and 	
\begin{align*}
	\langle \tau_{k_1}(\gamma_1)\dots \tau_{k_N}(\gamma_N)  \rangle_n^{\Hilb}&=\langle \tau_{k_1}(\gamma_1)\dots \tau_{k_N}(\gamma_N) \mid \emptyset\rangle_n^{\epsilon}, \quad &\text{ if } \epsilon \leq 1/n, \\
	\langle \psi^{k_1}\gamma_1\dots \psi^{k_m}\gamma_m\rangle_m^{\FuM}&= 	\langle \emptyset \mid \psi^{k_1}\gamma_1\dots \psi^{k_m}\gamma_m\rangle_m^{\epsilon}, &\quad  \text{ if } n=0. \hspace{0.35cm}
	\end{align*}

\end{defn}

\section{$I$-function}

\subsection{Definition of $I$-function} We will now define a quantity that measures the difference between spaces $\FuM^{\epsilon}_{n,[m]}(X)$ for different values of $\epsilon$. 	Consider a relative Hilbert scheme of the total space of tangent bundle $T_X \rightarrow X$, 
\[ \tau \colon  \Hilb_n(T_X/X) \rightarrow X,\] 
and the associated universal ideal  together with the canonical projections
\[\CI \rightarrow T_X\times_X \Hilb_n(T_X/X), \quad q_X \colon T_X\times_X \Hilb_n(T_X/X) \rightarrow X,\]
\[p_X \colon  T_X\times_X \Hilb_n(T_X/X) \rightarrow \Hilb_n(T_X/X). \]
There is a $\BC^*_z$-action on $T_X$ given by scaling the fibers. It induces a $\BC^*_z$-action on $\Hilb_n(T_X/X)$.  Let 
\[ 
V_n:=\Hilb_n(T_X/X)^{\BC^*_z}, \quad N^{\mathrm{vir}}
\]
be its $\BC^*_z $-fixed locus and the associated virtual normal complex. Let $\BC_{\mathrm{std}}$ be the weight-1 representation of $\BC^*_z$, whose class (the weight) we denote as follows
\[z:= e_{\BC^*_z}(\BC_{\mathrm{std}}).\]
We define the $\BC^*_z$-equivariant tautological classes, 
\[ \tau_k(\gamma):=p_{X^*}(\ch_{k}(\CI)\cdot q^*_X(\gamma)) \in H^*_{\BC^*_z}(\Hilb_n(T_X/X)). \] 
 Using the $\BC^*_z$-action and the (virtual) equivariant localisation \cite{GP}, we define the following localised $\BC^*_z$-equivariant class, called $I$-function, 
\begin{equation*}
I_n(z, \prod_i \tau_{k_i}(\gamma_i) ):=
e_{\BC^*_z}(T_X)\cdot \tau_*\left(\frac{\prod_i \tau_{k_i}(\gamma_i) \cap  [V_n]^\mathrm{vir}}{e_{\BC^*_z}(N^\mathrm{vir})}\right)  \in H^*(X)[z^\pm],
\end{equation*}
such that 
\[e_{\BC^*_z}(T_X)=\sum_i \mathrm{c}_i(X)z^{d-i}.\]
Note that we always expand rational functions in $z$ in the range $|z|>|a|>1$, e.g. 
\[\frac{1}{a-z}=\frac{1}{z}+\frac{a}{z^2}+\frac{a^2}{z^3}\dots, \]
where $a$ is some class. 
If $X$ carries an action of a torus $T$, by applying localisation with respect to $\BC_z^* \times T$,  we treat the expression above  as an element in the completed localised cohomology $H^*(X^T)[\![t_i^\pm, z^\pm]\!]$, where $t_i$ are weights of $T$.
\subsection{Universality of $I$-function}
An $I$-function is a tautological invariant associated to the relative Hilbert scheme  $\Hilb_n(T_X/X)$. It is important to stress that the enumerative geometry of the relative Hilbert scheme  $\Hilb_n(T_X/X)$ is essentially equivalent to the one of $\Hilb_n(\BC^d)$. This is made precise in the next proposition. 
\begin{prop} \label{univer} Let 
	\[\langle  \prod_i \tau_{k_i} \rangle^{\Hilb(\BC^d)}_n=\int_{[\Hilb_n(\BC^d)^A]} \frac{\prod_i \tau_{k_i}(\mathbb{1})}{e_A(N^{\mathrm{vir}})}  \in \BQ[a_1^{\pm},\dots, a_d^\pm]\]
	be an equivariant tautological integral on $\Hilb_n(\BC^d)$  with respect to the full torus $A=(\BC^*)^d$ with weights $a_i$, $i=1,\dots, d$. Let $a_i(X)$ be the Chern roots of $T_X$, then after the change of variables\footnote{By the properness of the associated fixed locus, $\langle  \prod_i \tau_{k_i} \rangle^{\Hilb(\BC^d)}_n$ does not have poles at weights $a_i-a_j$, hence the substitution is well defined.} $a_i\mapsto z+a_i(X)$, we have
	\begin{align*}
	I_n(z,  \prod_i \tau_{k_i}(\gamma_i) )&=e_{\BC^*_z}(T_X) \cdot  \langle \prod_i \tau_{k_i} \rangle^{\Hilb(\BC^d)}_n \cdot \prod_i \gamma_i. 
	\end{align*}

	\end{prop}
\begin{proof} Firstly, by the projection formula, we have 
\[ I_n(z,  \prod_i \tau_{k_i}(\gamma_i) )= I_n(z,  \prod_i \tau_{k_i}) \cdot \prod_i \gamma_i, \]
where we use a shortened notation, $\tau_{k_i}:=\tau_{k_i}(\mathbb{1})$. Hence it is enough to consider the descendents of $\mathbb{1}$. Assume firstly that $X$ does not carry an action of $T$. Consider the full flag variety,
\[ h\colon  \mathrm{Fl}(T_X) \rightarrow X .\] 
Recall that the pullback on the  cohomology, 
\[h^*\colon H^*(X) \rightarrow H^*(\mathrm{Fl}(T_X)),\]  is
injective, and there is a canonical locally-free filtration 
\[ 0 \hookrightarrow F_1 \hookrightarrow \dots \hookrightarrow F_d=h^*T_{X}, \]
such the associated graded pieces are line bundles $L_i$, $1 \leq i\leq d$. We now can inductively deform $F_i$ as a extension of $F_{i-1}$ by $L_i$  to a trivial extension, starting from $F_d$, by using the scalar scaling in the vector space of extensions.  
As a result, we obtain that the vector bundle
\[F=L_1\oplus\dots \oplus L_d\]
is deformation equivalent to $h^*T_X$, and admits a scaling action of a torus $A=(\BC^*)^d$. 

 Now consider the total space of $F$ without the zero section, which we denote by $Y$, together with the canonical projection $f \colon Y \rightarrow \mathrm{Fl}(T_X)$. The quotient $Y/A$ is isomorphic to $\mathrm{Fl}(T_X)$. Moreover, there is a $A$-equivariant identification of vector bundles on $Y$, 
 \[ f^*F\cong \BC^d_{(1,\dots,1)},\]
 where $\BC^d_{(1,\dots,1)}$ is a trivial vector bundle on $Y$ with $A$-weights $(1,\dots ,1)$. In fact, we treat $f^*F$ as a $\BC^*_z \times A$-equivariant vector bundle, where $\BC^*_z$ acts by scaling on $f^*F$ and trivially on the base $Y$.  This identification implies that a $\BC^*_z \times A$-equivariant tautological integral of  $\Hilb(f^*F/\mathrm{Fl}(X))$ has the following form, 
 \begin{multline*}
  \tau_*\left(\frac{\prod_i \tau_{k_i} \cap  [V_n]^\mathrm{vir}}{e_{\BC^*_z \times T}(N^\mathrm{vir})}\right) \\
  = \langle \prod_i \tau_{k_i}  \rangle^{\Hilb(\BC^d)}_n(z+a_1, \dots, z+a_d)  \cdot \mathbb{1}_{Y} \in H^*_A(Y)[(z+a_i)^\pm].  
  \end{multline*}
 Since $Y/T\cong \mathrm{Fl}(T_X)$, such that $a_i$ becomes Chern roots $\mathrm{c}_1(L_i)=a_i(X)$,  by the descending the expression above, we obtain that the same tautological integral on $\Hilb(F/\mathrm{Fl}(X))$  is equal to 
 \[  \langle \prod_i \tau_{k_i}  \rangle^{\Hilb(\BC^d)}_n(z+a_1(X), \dots, z+a_d(X)) \cdot \mathbb{1}_{\mathrm{Fl}(X)} \in H^*(\mathrm{Fl}(X))[z^\pm]. \]
 Vector bundles $F$ and $h^*T_X$ are deformation equivalent and virtual intersection theories are deformation invariant, we therefore arrive at the same conclusion for the relative Hilbert schemes $\Hilb(h^*T_X/\mathrm{Fl}(X))$. Finally, $h^*$ is injective on cohomology, hence we obtain the claim.   
 
 Assume $X$ has an action of a torus $T$. Recall that the equivariant cohomology can defined in terms of the products of $X$ and spaces with a free $T$-action, e.g.\ see \cite{EG2,EG} for the statement in  Chow groups. Moreover, all of the constructions involved in the proof, like flag manifolds, are naturally $T$-equivariant. By the same arguments, we thus obtain the conclusions above in the equivariant setting. Note, however, that the equivariant cohomology is not necessarily nilpotent, the final identity therefore takes place in  $H^*(X)[(z+a_i(X))^\pm]$. By passing to the localised cohomology, taking completion and expansion in the variable $z$, we obtain the identity in $H^*(X^T)[\![t_i^\pm, z^\pm]\!]$. 
 \end{proof}
\section{Wall-crossing for Hilbert schemes}

\subsection{Wall-crossing formula}Even for a wall-crossing formula which relates $\Hilb_n(X)$ and $\FuM_n(X)$, we need a much more general wall-crossing formula relating $\FuM_{n,[m]}^\epsilon(X)$ for different values of $\epsilon$.  

Let $\epsilon_0=1/n_0$ for some $n_0 \in \BZ$, such that $1\leq n_0\leq n$. We call such $\epsilon_0$ a $\textit{wall}$, because $\FuM_{n,[m]}^\epsilon(X)$ changes as we cross $\epsilon_0$. Let $\epsilon_+$ and $\epsilon_-$ be values of $\epsilon$ close to $\epsilon_0$ from the right and the left, respectively. We now state the wall-crossing formula, which compares tautological integrals of $\FuM_{n,[m]}^{\epsilon_-}(X)$ and $\FuM_{n,[m]}^{\epsilon_+}(X)$ for each wall $\epsilon_0$. 
\begin{thm} \label{maintheorem}
		If $d\leq 3$, we have a wall-crossing formula for each wall $\epsilon_0 \in \BR_{>0}$:
	\begin{multline*}	\langle \prod^{i=N}_{i=1} \tau_{k_i}(\gamma_i) \mid \prod^{j=m}_{j=1} \psi_j^{k'_j}\gamma'_j \rangle^{\epsilon_-}_n-\langle \prod^{i=N}_{i=1} \tau_{k_i}(\gamma_i) \mid \prod^{j=m}_{j=1} \psi_j^{k'_j}\gamma'_j \rangle^{\epsilon_+}_n\ \\
		=\sum_{k\geq 1}\sum_{ \underline{[N]}} \Big\langle \prod_{i \in N'} \tau_{k_i}(\gamma_i) \mid \prod^{j=m}_{j=1} \psi_j^{k'_j}\gamma'_j \cdot \prod^{\ell=k}_{\ell=1} I_{n_0}\Big( -\psi_{m+\ell}, \prod_{i\in N_\ell} \tau_{k_i}(\gamma_i) \Big)  \Big\rangle_{n'}^{\epsilon_+}/ k!,
	\end{multline*}
where we sum over integers $k$ and ordered partitions 
\[\underline{[N]} =(N', N_1,\dots, N_k)\]
 of the set $\{1,\dots,N\}$, such that subsets $N'$ and $N_\ell$ can be empty; $n_0=1/\epsilon_0$ and $n'=n-kn_0$.  We use the following convention for the substitution of the variable $z$ appearing in $I$-functions,  

\begin{equation*}
		[\mathrm{Class}]z^k \mapsto 
		\begin{cases}
			(-\psi)^k[\mathrm{Class}],& \text{if }k\geq 0,\\
			0,              & \text{otherwise}.
		\end{cases}	
	\end{equation*}
	
\end{thm}	

Note the non-polar part of $I$-functions is always a polynomial.  Hence the substitution of $z$ by $\psi$-classes is well-defined even in the presence of a torus $T$. The theorem is proved in Section \ref{secproof}. The next corollary is a direct consequence of the theorem. 

\begin{cor} \label{maincor}
		If $d\leq 3$, we have a wall-crossing formula, expressing tautological integrals on $\Hilb_n(X)$ in terms of $I$-functions and tautological integrals on $\FuM_{[n]}(X)$:
	\begin{multline*}	 \langle \tau_{k_1}(\gamma_1)\dots \tau_{k_N}(\gamma_N)\rangle^{\Hilb}_n	\\
		=\sum_{\underline{(n,[N])}} \Big\langle   I_{n_1}\Big( -\psi_1, \prod_{i\in N_1} \tau_{k_i}(\gamma_i) \Big) \dots  I_{n_k}\Big( -\psi_k, \prod_{i\in N_k} \tau_{k_i}(\gamma_i)\Big) \Big\rangle_k^{\FuM}/ k!,
	\end{multline*}
where we sum over ordered partitions
\[\underline{(n,[N])} =((n' ,N'), (n_1,N_1),\dots, (n_k,N_k))\]
of $n$ and $[N]$, such that parts of a partition of $n$ are strictly positive, while parts of a partition of $[N]$ can be empty.  We follow the same convention for the substitution of the variable $z$ as in Theorem \ref{maintheorem}.  
	\end{cor}

\begin{proof} The claim follows by applying the theorem inductively crossing all walls between $\epsilon<1/n$ and $\epsilon>1$, including $\epsilon_0=1$.  Note that since $\FuM^\epsilon_{n,[m]}(X)$ is empty for $\epsilon>1$, all insertions move to the right of the bar in the bracket. 
	\end{proof}

The corollary above is not an equivalence of tautological intersection theories of  $\Hilb_n(X)$ and $\FuM_{[n]}(X)$.  It is rather a statement about the universality of $\Hilb_n(\BC^d)$ which manifests itself in terms of $I$-functions, while $\FuM_{[n]}(X)$ provides corrections for  the comparison of tautological intersection theories of $\Hilb_n(\BC^d)$ and $\Hilb_n(X)$. 

\begin{cor} \label{unicor} If $d\leq 3$, the tautological integrals $\langle \tau_{k_1}(\gamma_1)\dots \tau_{k_N}(\gamma_N)\rangle^{\Hilb}_n$ universally depend only on $k_i$, and a finite number of intersection numbers of $\mathrm{c}_i(X^T)$, $\mathrm{c}_j(N_{X/X^T})$ and $\gamma_\ell$, 
	\[  \prod \mathrm{c}_{i_k}(X^T) \cdot \prod  \mathrm{c}_{j_k}(N_{X/X^T})  \cdot \prod \gamma_{\ell_k} .\]
	\end{cor}

\begin{proof} This follows from the universality of: 
	\begin{itemize} 
		\item the wall-crossing formula, Corollary \ref{maincor},
		\item the $I$-function, Proposition \ref{univer},
		\item the blow-up construction of  $\FuM_{[n]}(X)$, \cite{FM}. 
		\end{itemize}
	
Note that the $T$-fixed locus $\FuM_{[n]}(X)^T$ will depend on the dimensions of the components of $X^T$ and will consist of $k$-tuple of points lying  on  $X^T$ and on bubbles attached to $X^T$.  As such it will be a union of  products of Fulton--MacPherson spaces of lower dimensions, whose normal bundles are expressible in terms of $\psi$-classes and the normal bundle of $X^T$. The  fixed locus for zero-dimensional $X^T$ inside $2$-dimensional $X$ is considered in Section \ref{sector}.   By \cite{FM}, the cohomology ring of a Fulton--MacPherson spaces have a universal presentation in terms of cohomology of $X$ and boundary divisors, whose relations involve  Chern classes of $X$ and boundary divisors. 
\end{proof}
\section{Wall-crossing for $n$-fold products} \label{secnfold}

\subsection{Stability }
We also have a wall-crossing formula between $n$-fold products $X^n$ and Fulton--MacPherson spaces $\FuM_{[n]}(X)$. Since the intersection theory of $n$-fold products is easily accessible, it provides constraints for the tautological integrals on the side of $\FuM_{[n]}(X)$, which we use in Section \ref{sectioncomp}. We expect that all tautological integrals on $\FuM_{[n]}(X)$ can be computed via this wall-crossing formula. 

For the wall-crossing considered in this section, we need the following stability condition\footnote{In fact, this stability condition could be added on top of our $\epsilon$-weightedness, it would control the multiplicity of markings $\underline{p}=\{p_1,\dots,p_m\}$. }, considered by Hassett \cite{Has} in dimension one, and by  \cite{Rou} in arbitrary dimensions.
	\begin{defn}
	Given $\delta\in \BR_{>0}$.  Let $W$ be a FM degeneration together with two sets of ordered points, $\underline{x}=\{x_1,\dots,x_n\}$ and $\underline{p}=\{p_1,\dots,p_m\}$, such that the points in $\underline{p}$ are pairwise distinct.  The triple $(W, \underline{x}, \underline{p})$ is $\delta$-weighted, if 
		\begin{enumerate} 
			\item  $\underline{x} \cap  \underline{p}=\emptyset$,
			\item $\underline{x} \subset W^{\mathrm{sm}}$ and $\underline{p} \subset W^{\mathrm{sm}}$, 
			\item for all $x \in W$, $\mathrm{mult}_x(\sum_i x_i)\leq 1/\delta$, 
			\item  for all end components $\p^{d} \subset W$, such that $\underline{p} \cap \p^d=\emptyset$,  $\mathrm{mult}({\sum_i x_i}_{|\p^{d}})>1/\delta$, 
			\item the group of automorphisms $ \Aut(W, \underline{x}, \underline{p})$ is finite.
		\end{enumerate} 	
	\end{defn}
We define the moduli space of $\delta$-weighted triples $(W, \underline{x},\underline{p})$, 
\[  
\FuM^{\delta}_{[n,m]}(X):=\{\delta\text{-weighted } (W,\underline{x}, \underline{p}) \} / \sim.
\] 

Assume $m=0$. If $\delta \leq 1/n$, the stability prohibits all end components, hence 
\[ \FuM^\delta_{[n]}(X)=X^n, \quad \text{if } \delta \leq 1/n.\] 
If $\delta=1$,  multiplicities of points can be at most $1$,  
\[ \FuM^\delta_{[n]}(X)=\FuM_{[n]}(X), \quad \text{if } \delta=1.\]
If $\delta>1$, then $ \FuM^\delta_{[n]}(X)$ is empty, 
\[ \FuM^\delta_{[n]}(X)=\emptyset , \quad \text{if } \delta>1.\]

\subsection{Tautological integrals} 
Tautological classes on $\FuM^{\delta}_{[n,m]}(X)$ are very similar to those of $\FuM_{[n]}(X)$. Firstly, we have projection maps associated to markings $\underline{x}$,

\begin{align*}
	\pi_i =q_\CW \circ s_i \colon \FuM_{[n,m]}^\delta(X)  \rightarrow X, \quad i=1,\dots, n.
\end{align*}
as well as  $\psi$-classes,
\[ \BL_{i}=s^*_iT_{p_\CW}^\vee,\quad  \psi_i:= (\mathrm{c}_1(\BL_i)-\pi_i^*\mathrm{c}_1(X))/d, \quad i=1,\dots,n. \]
As in the case of  $\FuM^{\epsilon}_{n,[m]}(X)$, we also have projection maps and $\psi$-classes associated to markings $\underline{p}$, we will denote them simply by $\pi_{n+j}$ and $\psi_{n+j}$ for $j=1, \dots, m$. Note that for $\delta \leq 1/n$ and $m=0$, $\psi$-classes vanish, while $\tilde{\psi}$-classes are equal to $\mathrm{c}_d(\pi_i^*T^\vee_X)$.

Spaces $\FuM_{[n,m]}^\delta(X)$ are smooth for all $\delta$, hence we will use the standard fundamental classes to define the associated tautological integrals.

\begin{defn} Let $\gamma_i,\gamma_j' \in H^*(X)$ be some classes, \textit{tautological integrals} on $\FuM_{n,[m]}^\delta(X)$ are defined as follows, 
\begin{multline*}
	\langle \psi_1^{k_1}\gamma_1\dots \psi_n^{k_n}\gamma_n \mid \psi_{n+1}^{k'_1}\gamma'_1\dots \psi_{n+m}^{k'_m}\gamma'_m\rangle_n^{\delta}:=\\ 
	\int_{[\FuM_{[n,m]}^\delta(X)] }\prod_{i} \psi_i^{k_i}\pi_i^*\gamma_i \prod_j \psi^{k'_j}_{m+j}\pi^*_{m+j}\gamma'_j.
\end{multline*}  
We use similar conventions as in Definition \ref{integrals1}.

\end{defn}

\subsection{Wall-crossing} We now define the $I$-function for the $\delta$-weightedness. 	Consider the relative $n$-fold product of $T_X \rightarrow X$, 
\[ \tau\colon (T_X/X)^n:=T_X\times_X \dots \times_X T_X \rightarrow X.\] 
The $\BC^*_z$-action on $T_X$ given by scaling the fibers induces a $\BC^*_z$-action on $(T_X/X)^n$. 
Using the $\BC^*_z$-action and equivariant localisation, we define the following $\BC^*_z$-equivariant class, called the $n$-fold  $I$-function, 
\[I_{X^n}(z, \prod_{i} \psi_i^{k_i}\gamma_i ):=e_{\BC^*_z}(T_X)\cdot \tau_*\left( \prod_{i} \psi_i^{k_i} \pi^*_i\gamma_i \right)  \in H^*(X)[z^\pm].\]
In what follows, we use the same notation and conventions as in Theorem \ref{maintheorem}. 
\begin{thm} \label{maintheorem2}
	For all $d$, we have a wall-crossing formula for each wall $\delta_0 \in \BR_{>0}$:
	\begin{multline*}	\langle \prod^{i=n}_{i=1} \psi_i^{k_i}\gamma_i \mid \prod^{j=m}_{j=1} \psi_{n+j}^{k'_j}\gamma'_j \rangle^{\delta_-}_n-\langle \prod^{i=n}_{i=1} \psi_i^{k_i}\gamma_i \mid \prod^{j=m}_{j=1} \psi_{n+j}^{k'_j}\gamma'_{j} \rangle^{\delta_+}_n\ \\
		=\sum_{ \underline{[n]}} \Big\langle \prod_{i \in N'} \psi_i^{k_i}\gamma_i \mid \prod^{j=m}_{j=1} \psi_{n+j}^{k'_j}\gamma'_j \cdot \prod^{\ell=k}_{\ell=1} I_{X^{n_\ell}}\Big( -\psi_{n+m+\ell}, \prod_{i\in N_\ell} \psi_i^{k_i}\gamma_i \Big)  \Big\rangle_{n'}^{\delta_+}/ k!,
	\end{multline*}
	where we sum over  ordered partitions $\underline{[n]}=\{N', N_1 \dots, N_k\}$ of the set $[n]$, such that 
	\[ |N_\ell|=n_0= 1/\delta_0, \quad \text{for all }\ell\geq 1,\] 
	and $|N'|=n'$. 
\end{thm}

\begin{cor} \label{maincor2}
	For all $d$, we have a wall-crossing formula:
	\begin{multline*}	 \langle \psi_1^{k_1} \gamma_1\dots \psi_n^{k_n}\gamma_n\rangle_{X^n}	\\
		=\sum_{\underline{[n]}} \Big\langle   I_{X^{n_1}}\Big( -\psi_1, \prod_{i\in N_1} \psi_i^{k_i} \gamma_i \Big) \dots  I_{X^{n_k}}\Big( -\psi_k, \prod_{i\in N_k} \psi_i^{k_i}\gamma_i\Big) \Big\rangle^{\FuM}_{k}/ k!,
	\end{multline*}
	where we sum over  ordered partitions of $[n]$, such that parts of a partition are non-empty.  
\end{cor}

\section{Euler characteristics} \label{sectioncomp}

\subsection{Topological Euler characteristics} We will now illustrate how the wall-crossing formula works by applying it to compute the 
 topological (virtual) Euler characteristics of Hilbert schemes of an arbitrary variety, knowing the answer for affine spaces. This shows an  essential feature of the wall-crossing formula - it involves integrands of very different nature. 

 If $d\leq 2$, this amounts to applying the wall-crossing to the Euler class of relative tangent bundle, 
\[ 
e_{\mathrm{rel}}=e(T_{\FuM^\epsilon_{n,[m]}(X)/\CF\CM_{[m]}(X)}), 
\]
which is expressible in terms of descendents.  Since the Euler class is multiplicative and the tangent bundle of a product is the sum of tangent bundles of its factors, the wall-crossing formula Corollary \ref{maincor} takes a particularly simple form\footnote{Equivalently, we can endow $\FuM^\epsilon_{n,[m]}(X)$ with an obstruction theory whose obstruction bundle is the relative tangent bundle $T_{\FuM^\epsilon_{n,[m]}(X)/\CF\CM_{[m]}(X)}$, then the virtual fundamental class is $e_{\mathrm{rel}}$, and the wall-crossing formula compares the integrals of $1$ against $e_{\mathrm{rel}}$. }, 
\begin{equation}	\label{wallcross1}
	 e(\Hilb_n(X))=\langle e_{\mathrm{rel}} \rangle^{\Hilb}_n	\\
	=\sum_{\underline{n}} \Big\langle   I_{n_1}( -\psi_1,e_{\mathrm{rel}} ) \dots  I_{n_k}( -\psi_k, e_{\mathrm{rel}}) \Big\rangle_k^{\FuM}/ k!,
\end{equation}
where we sum over ordered partitions $\underline{n}=(n_1,\dots, n_k)$ of $n$, such that $n_i>0$. The integrand of the $I$-function is the Euler class of the $X$-relative tangent bundle of $\Hilb_n(T_X/X)$, which we also denote by $e_{\mathrm{rel}}$.

If $d=3$,  we apply the wall-crossing formula to the integral of $1$, because the virtual dimension is $0$, 
\begin{equation} \label{wallcross2}
	e_{\mathrm{vir}}(\Hilb_n(X))=\langle \emptyset \rangle^{\Hilb}_n	\\
	=\sum_{\underline{n}} \Big\langle   I_{n_1}( -\psi_1,\emptyset ) \dots  I_{n_k}( -\psi_k, \emptyset) \Big\rangle_k^{\FuM}/ k!.
\end{equation}
We will explicitly evaluate the formulas above, starting with curves. 

\subsection{Dimension one}

Assume $d=1$. 
The relevant $I$-function in this case has the following expression, 
\[
I_n(z,e_{\mathrm{rel}})= (z+\mathrm{c}_1(X))\cdot \tau_*e( T_{\Hilb_n(T_X/X)/X}) \in H^*(X)[z^\pm],
\]
where, abusing the notation, $e( T_{\Hilb_n(T_X/X)/X})$ denotes the localised Euler class. 
Since $\Hilb_n(T_X)^{\BC^*_z}=X$, we obtain that $\tau_*e( T_{\Hilb_n(T_X/X)/X})=\mathbb{1}_X$. Hence for $n\geq 1$, we have
\[
I_n(z,e_{\mathrm{rel}})=(z+\mathrm{c}_1(X)) \in H^*(X)[z^\pm].
\]
We sum over $n$,
\[
I(q,z):=\sum_{n\geq 1} I_n(z, e_{\mathrm{rel}})q^n= \left(\frac{1}{1-q}-1\right) \cdot (z+\mathrm{c}_1(X)).
\]
We now plug it into the wall-crossing formula (\ref{wallcross1}), 

\begin{multline} \label{sum1}
	\sum_{n\geq 1} e(\Hilb_n(X))q^n \\
	=\sum_{k\geq 1} \left(\frac{1}{1-q}-1\right)^k\cdot \Big\langle \left( \mathrm{c}_1(X)-\psi_1 \right) \dots  \left(\mathrm{c}_1(X)-\psi_k \right) \Big\rangle^{\FuM}_k/k!. 
\end{multline}
Note that in our notation,   $\psi_i-\pi^*_i\mathrm{c}_1(X)=\tilde{\psi}_i$, where $\tilde{\psi}_i$ is the standard $\psi$-class associated to a marking of a curve. To evaluate the expression above, we therefore have to compute the invariant $\langle -\tilde{\psi}_1 \dots -\tilde{\psi}_k \rangle^{\FuM}_k/k!$, which can be readily computed using the dilaton equation, Lemma \ref{dilaton},
\[
 \langle -\tilde{\psi}_1 \dots -\tilde{\psi}_k \rangle^{\FuM}_k/k!=\binom{\mathrm{c}_1(X)}{k}= \binom{2-2g(X)}{k}.
\]
Using the binomial formula, we arrive at the expression for (\ref{sum1}), 
\[
1+\sum_{n\geq 1} e(\Hilb_n(X))q^n= 1+\sum_{k\geq 1} \binom{\mathrm{c}_1(X)}{k} \left(\frac{1}{1-q}-1\right)^k =\left(\frac{1}{1-q}\right)^{\mathrm{c}_1(X)},
\]
which is Macdonald's formula for the Euler characteristics of symmetric products of a curve, \cite{Mac}. 

\subsection{Dimension two}
	Assume now $d=2$. Consider the $I$-function, 
	 \[ I_n(z,e_{\mathrm{rel}})= \left(\mathrm{c}_2(X)+z\mathrm{c}_1(X)+z^2\right)\cdot \tau_*e( T_{\Hilb_n(T_X/X)/X}) \in H^*(X)[z^\pm]. \] 
	The fixed locus $\Hilb_n(T_X/X)^{\BC^*_z}$ is a smooth fibration over $X$, whose fiber is $\Hilb_n(\BC^2)^{\BC^*_z}$. In particular, 
	\[
	\tau_*e( T_{\Hilb_n(T_X/X)/X})=e(\Hilb_n(\BC^2))\cdot \mathbb{1}_X,
	\]
	which also follows from Proposition \ref{univer}. 
The generating series of the topological Euler characteristics of $\Hilb_n(\BC^2)$ is given by the reciprocal of Euler function. Summing over $n$, we therefore obtain,
	
	\begin{align*}
	I(q,z):=\sum_{n\geq 1}I_n(z,e_{\mathrm{rel}})q^n
	&=(f(q)-1) \left(\mathrm{c}_2(X)+z\mathrm{c}_1(X)+z^2\right), \\
	f(q)&=\left(\prod_{m\geq 1}\frac{1}{1-q^m}\right).\
	\end{align*}
	
	Plugging this expression into the wall-crossing formula, we obtain
	\begin{equation*}
		\sum_{n\geq 1} e(\Hilb_n(X))q^n 
		= \sum_{k\geq 1}(f(q)-1)^k\cdot \Big \langle \prod_{i}\left( \mathrm{c}_2(X)-\psi_i\mathrm{c}_1(X)+\psi_i^2\right) \Big \rangle^{\FuM}_k/k!. 
	\end{equation*} 
Recall that $\left( \mathrm{c}_2(X)-\psi_i\mathrm{c}_1(X)+\psi_i^2\right)=\tilde{\psi}_i$. By Lemma \ref{dilaton}, there is a higher-dimensional analogue of the dilaton equation, which gives us
	\[\langle \tilde{\psi}_1 \dots \tilde{\psi}_k \rangle^{\FuM}_k/k! =\binom{\mathrm{c}_2(X)}{k}.\]
Using the wall-crossing and binomial formulas, we therefore conclude
	\[1+\sum_{n\geq 1} e(\Hilb_n(X))q^n=f(q)^{\mathrm{c}_2(X)}, \]
	which is G\"ottsche's formula for the topological Euler characteristics of Hilbert schemes of points on a surface \cite{Gott}. 
	\subsection{Dimension three}  Finally, assume $d=3$. Dimension three requires a bit more work. Firstly, by Proposition \ref{univer}, to determine the $I$-function, it is enough to consider the associated tautological integral on $\BC^3$. By  \cite[Theorem 1]{MNOP2}, we have
	 	\begin{align*}
	 	 1+\sum_{n\geq 1}\langle \emptyset\rangle^{\Hilb(\BC^3)}_n= M(-q)^{\frac{\mathrm{c}^A_3(T_{\BC^3} \otimes K_{\BC^3})}{ c^{A}_3(T_{\BC^3})}} ,  \quad 	M(q)=\prod_{m\geq 1}\frac{1}{(1-q^m)^m}.
	 \end{align*}
By making the substitution  $ a_i \mapsto z+a_i(X) $,   Proposition \ref{univer} gives  us that
	
	\begin{equation*}
	 I(q,z)=\sum_{n\geq 1} I_n(z,\emptyset)q^n = e_{\BC^*_z}(T_X)\cdot \bigg( M(-q)^{\frac{\mathrm{c}^{\BC^*_z}_3(T_X \otimes K_X)}{\mathrm{c}^{\BC^*_z}_3(T_X)}}-1\bigg).
	\end{equation*}


	It remains to evaluate the integrals on FM spaces appearing on the right-hand side of the wall-crossing formula (\ref{wallcross2}), 
		\begin{equation} \label{Hilb3}
		\sum_{n\geq 1} e_{\mathrm{vir}}(\Hilb_n(X))q^n=\sum_{k\geq 1}\langle I(q,z)\dots I(q,z) \rangle^{\FuM}_k/k!.
	\end{equation}
	  It would be easy, if $\mathrm{c}^{\BC^*_z}_3(T_X \otimes K_X)$ did not have the twist by the canonical bundle $K_X$. As a consequence of this twist, the expression is highly complicated, so computing it directly should rather be avoided. 
	  
	  Luckily, there is a way to avoid  direct computations. The reason is that the same invariants appear in the wall-crossing formula between Fulton--MacPherson spaces and $n$-fold products, when we apply Corollary \ref{maincor2} to $\mathrm{c}_3(\BL_i^\vee \otimes \det(\BL_i))$. On the side of $n$-fold products,  these  are the following integrals,
	  \[ \langle \mathrm{c}_3(\BL_1^\vee \otimes \det(\BL_1)) \dots  \mathrm{c}_3(\BL_n^\vee \otimes \det(\BL_n)) \rangle_{X^n}=\mathrm{c}_3(T_X\otimes K_X)^n,\]
	  note that the exponent $n$ indicates the power of $\mathrm{c}_3(T_X\otimes K_X)$ as a number (or an element in $\BQ[t_i^\pm])$, not as a class. In fact, we are interested in the case when points are unordered, this amounts to descending the above expression to $S_n(X)=[X^n/S_n]$ or, more simply, dividing by $n!$.  Hence the relevant $I$-function is 
	\begin{align*}
	I^{S}_n(z, \prod^n_{i=1}\mathrm{c}_3(\BL_i^\vee \otimes \det(\BL_i)) )&:=I_{X^n}(z, \prod^n_{i=1}\mathrm{c}_3(\BL_i^\vee \otimes \det(\BL_i)) )/n!\\
	&=e_{\BC^*_z}(T_X)\cdot \left(\frac{\mathrm{c}^{\BC^*_z}_3(T_X\otimes K_X)}{\mathrm{c}^{\BC^*_z}_3(T_X)}\right)^n/n! \in H^*(X)[z^\pm].
	\end{align*}
	Summing over $n$, we obtain 
	\begin{multline*}
	 I^{S}(q,z)= \sum_{n\geq 1}I^{S}_n(z,\prod^n_{i=1}\mathrm{c}_3(\BL_i^\vee \otimes \det(\BL_i))  )q^n \\
	 =e_{\BC^*_z}(T_X)\cdot \left( \exp \left(\frac{\mathrm{c}^{\BC^*_z}_3(T_X\otimes K_X)}{\mathrm{c}^{\BC^*_z}_3(T_X)}q\right ) -1 \right). 
	\end{multline*}
 The wall-crossing formula in Corollary (\ref{maincor2}), after removing the order on the  points, gives us the following identity, 
	\begin{multline} \label{Sym3}
1+ c_3(T_X\otimes K_X)^n q^n/n!=\exp( c_3(T_X\otimes K_X)q) \\
	=1+\sum_{k\geq 1}\langle I^{S}(q,z)\dots I^{S}(q,z) \rangle^{\FuM}_k/k!.
	\end{multline}

Observe now a simple relation between $I$-functions of Hilbert schemes and Symmetric products, 
\begin{equation} \label{subs}
	I^{S}(q',z)= I(q,z), \quad q'\mapsto\mathrm{log}(M(-q)).
\end{equation}
Hence the substitution of variables  (\ref{subs}) turns invariants on the right-hand side of (\ref{Sym3}) into invariants that compute the topological virtual Euler characteristics of Hilbert schemes in (\ref{Hilb3}),
\begin{align*}
\langle I^{S}(q',z)\dots I^{S}(q',z) \rangle^{\FuM}_k&= \langle I(q,z)\dots I(q,z) \rangle^{\FuM}_k \\
q'&\mapsto \mathrm{log}(M(-q)).
\end{align*}
Overall, by (\ref{Hilb3}) and (\ref{Sym3}), we obtain 
\begin{multline*}
1+\sum_{n\geq 1} e_{\mathrm{vir}}(\Hilb_n(X))q^n=\exp( \mathrm{c}_3(T_X\otimes K_X)\log(M(-q)))\\
=M(-q)^{\mathrm{c}_3(T_X\otimes K_X)},
\end{multline*}
this provides another proof of Li's and Levine--Pandharipande's calculations \cite{LiDT,LP}.

\section{Affine plane}	

\subsection{Tautological vector bundle}In light of Proposition \ref{univer}, it natural to ask what role does a wall-crossing play for affine spaces $\BC^d$ themselves?  We will answer this question for $X=\BC^2$ with a diagonal action of a one-dimensional torus $T=\BC^*$, whose weight we denote by $t$.  Instead of descendents of the type $\tau_{k}$, we will consider more well-behaved Chern characters of the tautological bundle,
\[ \ch_k := \ch_k (\CV),\]
where $\CV=p_*(\CO_\CZ)$ for the universal subscheme $\CZ \subset \BC^2 \times \Hilb_n(\BC^2)$ and the projection $p \colon \BC^2 \times \Hilb_n(\BC^2) \rightarrow \Hilb_n(\BC^2)$. Descendents $\tau_{k}$ and $\ch_k$ are related by the Grothendieck--Riemann--Roch theorem. The wall-crossing for $\ch_k$ takes exactly the same form as for $\tau_k$, since they satisfy the same splitting property for the degenerations (\ref{split}). The weights of the tautological bundle at the torus-fixed points are computed in the proof of \cite[Proposition 5.8]{Nak}.

\subsection{Wall-crossing for the affine plane}  \label{secaff} 
Firstly, observe the following relation between tautological integrals of $\Hilb_n(\BC^2)$, \[\langle \prod_{i}\ch_{k_i} \rangle^{\Hilb}_n(t) \in \BQ[t^\pm],\]  which we treat as homogenous Laurent polynomials in the variable $t$, and its $I$-functions, 
\begin{equation} \label{ifunction} e_{\BC^*_z}(T_{\BC^2})=(t+z)^2, \quad I(z, \prod_i \ch_{k_i})= (t+z)^2 \cdot \langle \prod_{i}\ch_{k_i} \rangle^{\Hilb}_n(t+z),
	\end{equation}
the expression above follows directly from the definition of $I$-functions (and also from Proposition \ref{univer}). 

Secondly,  the wall-crossing formula is insensitive to the poles of $I$-functions, as is stated in Theorem \ref{maintheorem}. More specifically, if 
\begin{equation} \label{inequality}
\sum_i k_i < 2n-2, 
\end{equation}
then by (\ref{ifunction})  the $I$-function $I(z, \prod_i \ch_{k_i})$ does not contribute to the wall-crossing, as its non-polar part is zero. In practice this means that if apply the wall-crossing to $\langle \prod_{i}\ch_{k_i} \rangle^{\Hilb}_n$ for values of $k_i$ satisfying the inequality above, then it will involve only $I$-functions for a strictly smaller number of points.

Lastly, the $T$-fixed locus of $\FuM_{[n]}(\BC^2)$ admits a particularly nice description. It is not difficult to see that it is the locus of marked FM degenerations, such that all marked points lie on a tree of bubbles attached to $0\in \BC^2$.   Essentially, this locus is a higher-dimensional analogue of moduli spaces of stable genus-0 curves, $\Mbar_{0,n+1}$, and  was considered in \cite{CGK}. Following \textit{loc.\ cit.}, we will denote it by $T_{n}$. By construction of $T_{n}$, the Euler class of its normal bundle inside $\FuM_{[n]}(\BC^2)$ is given by
\begin{align*}
 &t^2(t-\psi_\infty), \quad \text{if } n\geq 2, \\
 &t^2, \quad \hspace{1.5cm}\text{if } n=1,
 \end{align*}
where $-\psi_\infty$ is defined to be the restriction of the divisor $D_{[n]}=D_{\{1,\dots, n\}}$ to $T_{n}$ (see Section \ref{comp} for the definition of this divisor), geometrically, it arises as the obstruction to smooth out the tree of bubbles at $0\in \BC^2$; the factor $t^2$ comes from the obstruction to move the tree of bubbles from $0$ to another point on $\BC^2$.  We define $\psi$-classes on $T_n$ by using the line bundles $\BL_i$, 
\[ \text{on } T_n,\quad  \psi_i:= \mathrm{c}_1(\BL_i)/2, \quad i=1,\dots ,n.\]
Since $\mathrm{c}_1(\BC^2)=2t$, by our definition of $\psi$-classes in  Section \ref{secdescFM}, we obtain  the following relation between $\psi$-classes on $T_n$ and $\FuM_{[n]}(\BC^2)$,
\begin{equation} \label{psion}
 \psi_{i|T_n}=\psi_i-t.
 \end{equation}
  Hence the $T$-localised integrals on $\FuM_{[n]}(\BC^2)$ take the following form for $n\geq 2$, 
\[ \langle \psi_1^{{k_1}} \dots \psi_n^{k_n} \rangle^{\FuM}_{n}= \int_{T_{n}} \frac{(\psi_1-t)^{{k_1}} \dots (\psi_n-t)^{k_n}}{t^2(t-\psi_\infty)}, \]
while for $n=1$, it just an integral on $\BC^2$, for which the $\psi$-class vanishes. 
\subsection{Computations}
Overall,  Theorem \ref{maintheorem} allows us to express integrals $\langle \prod_{i}\ch_{k_i} \rangle^{\Hilb}_n$ with poles of order at least $3$ in the variable $t$ in terms of integrals $\langle \prod_{i \in N_\ell}\ch_{k_i} \rangle^{\Hilb}_{n'}$ for $n' <n$ and $N_\ell \subseteq \{1,\dots, N\}$, and integrals on $T_{n}$.  Let us illustrate how it works for  $\langle \ch_k \rangle ^{\Hilb}_n$.   We start with the simplest integral, 
\[\langle 1 \rangle_n^{\Hilb} \in \BQ[t^\pm].\] 
By (\ref{ifunction}), the associated $I$-functions are polar for $n\geq 2$, while for $n=1$, it equal to the identity, 
\[ I_1(z, 1)=\mathbb{1}. \]
Hence the wall-crossing formula gives us 
\[ \langle 1 \rangle_n^{\Hilb}=  \langle \mathbb{1} \dots \mathbb{1}\rangle^{\FuM}_n /n!=\frac{1}{n!t^{2n}},\]
where the second equality is obtained by pushing forward the identity using the natural projection $\FuM_{[n]}(\BC^2) \rightarrow \BC^{2n}$. The equality above is well-known to experts and can be proved by using the Hilbert--Chow map $\Hilb_n(\BC^2) \rightarrow \BC^{2n}/S_n$.

The next  integral we consider, 
\[ \langle \ch_1 \rangle ^{\Hilb}_n \in \BQ[t^\pm],\]
takes an even simpler form. In this case, again by (\ref{ifunction}), there are two non-polar $I$-functions, 
\[ I_1(z, 1)=\mathbb{1}, \quad I_1(z, \ch_1)=0, \]
where the second identity is due to the fact the first Chern class of the diagonal in $\BC^2 \times \BC^2$ is equivariantly trivial. The wall-crossing formula therefore gives us that
\[\langle \ch_1 \rangle ^{\Hilb}_n= \sum^{n}_{i=1}\langle \mathbb{1} \dots I_1(-\psi_i,\ch_1) \dots  \mathbb{1}\rangle^{\FuM}_n /n! =0. \]

We continue with $\langle \ch_2 \rangle^{\Hilb}_n$. There are three non-polar $I$-functions, 
\[I_1(z,1)=\mathbb{1}, \quad I_1(z,\ch_2)=0, \quad I_2(z,\ch_2)=-\mathbb{1}/4,\]
where the last equality follows from the torus-weights expression from \cite[Proposition 5.8]{Nak2} applied to the full torus $(\BC^*)^2$ and then specialised to the diagonal one, see also \cite[Section 1.1.2]{Arb}. We thus obtain that for $n\geq 3$, 
\begin{align*}
\langle \ch_2 \rangle ^{\Hilb}_n&=\sum^{n-1}_{i=1}\langle \mathbb{1} \dots I_2(-\psi_i,\ch_2) \dots  \mathbb{1}\rangle^{\FuM}_{n-1} /{(n-1)}! \\
&=-\frac{\langle \mathbb{1} \dots    \mathbb{1}\rangle^{\FuM}_{n-1}}{ 4\cdot  (n-2)!}= -\frac{1}{4\cdot (n-2)!t^{2(n-1)}}.
\end{align*}
While for $n=2$, the wall-crossing formula gives a tautology, 
\[ \langle \ch_2 \rangle ^{\Hilb}_2 =\langle \ch_2 \rangle ^{\Hilb}_2. \]
This is general feature of the wall-crossing for the affine plane with a one-dimensional torus: if we try to apply it integrals that do not satisfy  the equality (\ref{inequality}), then it will always give the tautology\footnote{If we consider the fully equivariant affine plane, this does not seem to be the case, because then equivariant invariants might have poles in equivariant variables. }  as above, because the degree of resulting insertions on $T_n$ will be too high with respect to the dimension of $T_n$, unless $n=1$, in which case the $\psi$-class vanishes. 

To sum up the discussion above,  let us state the result that allows to compute $\langle \ch_k \rangle ^{\Hilb}_n$ for $k< 2n-2$ in terms $\langle \ch_k \rangle ^{\Hilb}_{n'}$ for $2n' \leq k+2$ and some simple integrals on $T_n$. A similar result also applies to integrals with more insertions $\langle \prod_i \ch_{k_i} \rangle ^{\Hilb}_n$.  We define 
\begin{align*}
 I(q, z,\ch_k)&:= \sum_{1 \leq 2n\leq k+2} z^2 \langle \ch_k \rangle ^{\Hilb}_nq^n \in \BQ[z^\pm][\![q ]\!],\\
 \langle \ch_k \rangle ^{\Hilb}&:= \sum_{n\geq 1} \langle\ch_k \rangle ^{\Hilb}_nq^n  \in \BQ[t^\pm][\![q ]\!].
 \end{align*}
 A direct consequence of Theorem \ref{maintheorem} and the analysis above is the  following corollary. 
\begin{cor} \label{strangecor} We have
	 \[  \langle \ch_k \rangle ^{\Hilb}=\sum_{n\geq 1}\int_{T_{n}} \frac{I(q, -\psi_1,\ch_k)}{t^2(t-\psi_\infty)}\cdot \frac{q^{n-1}}{(n-1)!}.  \]
	\end{cor}

Note that $t$ in $(\ref{psion})$ is cancelled with $t$  in $I$-function; and the choice of the marking with which we insert the $I$-function is reflected in the extra factor of $n$.  

Let us now say a few words about the integrals on $T_n$ on right-hand side of the corollary. By noticing that $T_{n}=D_{[n]} \cap \pi_1^{-1}(0)$, they can be easily evaluated by the string equation\footnote{Since $T_{n}=D_{[n]} \cap \pi_1^{-1}(0)$, all Chern classes vanish after restricting $\tilde{\psi_1}$ to $T_n$, making it equal to the restriction of $\psi^2_1$}, Lemma \ref{string}. The same proof applies to the string equation of $\psi_\infty$ with the difference that there is no sign\footnote{Roughly speaking, this is because the normal bundle of a divisor at infinity is $\CO(1)$, while for the exceptional divisor it is $\CO(-1)$. } by \cite[Lemma 6.0.18]{CGK}. The relevant integrals on the side of $T_{n}$ therefore take the following form, 
 \[ \int_{T_{n}} \psi_1^{k_1} \cdot \psi_\infty^{k_2}=(-1)^{\ceil{k_1/2}}\binom{n-2}{\floor{k_1/2},\floor{k_2/2}}, \quad k_1+k_2=2n-3, \]
 where we remove squares of $\psi$-classes until only one $\psi$-class on $T_2$ remains, which then can be computed by the identity (\ref{f1}) and by  \cite[Proposition 6.0.12]{CGK}, 
 \[ \int _{T_2}\psi_1=-1, \quad  \int _{T_2}\psi_\infty=1,\]
  the binomial coefficient comes from the consecutive application of the string equation.
 The corollary thus provides an effective way to compute integrals of $\ch_k$ for all $\Hilb_n(\BC^2)$,  knowing the answer for finitely many $n$ such that $2n\leq k+2$ (which can be determined by localisation). We, however, stress that for  $2n\leq k+2$, the corollary gives a tautology. Here is a list of formulas  for the first few integrals including those we computed above,
 \begin{align*}
 	 \langle 1 \rangle ^{\Hilb}&= \exp\left(\frac{q}{t^2}\right) \\
 	 \langle \ch_1 \rangle ^{\Hilb}&= 0 \\
 	  \langle \ch_2 \rangle ^{\Hilb}&=-\frac{q^2}{4t^2}\cdot \exp\left(\frac{q}{t^2}\right) \\
 	   \langle \ch_3 \rangle ^{\Hilb}&=\frac{q^2}{6t}\cdot \exp\left(\frac{q}{t^2}\right)\\   
 	    \langle \ch_4 \rangle ^{\Hilb}&=-\frac{q^2}{16} +\sum_{n\geq 4} \frac{1}{16}\frac{(n-3)q^n}{(n-2)!t^{2(n-2)}} -\frac{5q^3}{144t^2}\cdot\exp\left(\frac{q}{t^2}\right) \\
 	    \langle \ch_5 \rangle ^{\Hilb}&=\frac{q^2t}{60} - \sum_{n\geq 4} \frac{1}{60}\frac{(n-3)q^n}{(n-2)!t^{2(n-2)-1}}  -\frac{q^3}{60t}\cdot\exp\left(\frac{q}{t^2}\right)  \\
 	     \langle \ch_6 \rangle ^{\Hilb}&=-\frac{q^2t^2}{288}+\frac{77q^3}{4320}-\sum_{n\geq 5} \frac{77}{4320}\frac{(n-4)q^n}{(n-3)!t^{2(n-3)}}\\
 	     &-\sum_{n\geq 5}\frac{1}{576} \frac{q^n}{(n-2) \cdot (n-5)! t^{2(n-3)}}+\frac{77q^4}{4320t^2}\cdot \exp\left(\frac{q}{t^2}\right). 
 	\end{align*}
 
They were verified for small values of $n$ on a computer.  It was communicated to the author that the integrals considered above may also be attainable from the $K$-theoretic calculations of \cite{Arb,Arb2} and \cite{WZ}. 
Although, it is unclear  if this will give the same  formulas. 
 
 Unfortunately, our methods (at least, in their current form) cannot be applied to compute integrals of $\ch_k$ against the Euler class of $\Hilb_n(\BC^2)$ considered in \cite{Ok}. This is due to the same reason that Corollary \ref{strangecor} gives a tautology for $2n\leq 2k+2$, i.e.\ the $I$-functions are of too high degree for all Fulton--MacPherson correction terms except the first one, which gives back the invariant we try to compute because the $\psi$-class vanishes on $\FuM_{[1]}(\BC^2)=\BC^2$. 
\subsection{Torus fixed locus of $\FuM_{[n]}(\BC^2)$ and 
$\Mbar_{0,n+1}$} \label{sector} Of course, integrals on $\Hilb(\BC^2)$ equivariant with respect to the full torus $T=(\BC^*)^2$ are more appealing. In fact, using Corollary \ref{maincor} together with Corollary \ref{maincor2}, one can refine all of the formulas above to $(\BC^*)^2$-equivariant ones without analysing the $(\BC^*)^2$-fixed points of $\FuM_{[n]}(\BC^2)$. We, however, postpone it to the future work, ending this section with the analysis of $(\BC^*)^2$-fixed points of $\FuM_{[n]}(\BC^2)$ instead. We find it very interesting on its own, as it is expressible in terms of moduli spaces of marked genus-zero curves.


 There are two natural subloci of $\FuM_{[n]}(\BC^2)$ given by FM compactifications on the $x$- and $y$-axis of $\BC^2$, which we denote by $\BC_{0}$ and $\BC_{1}$. More precisely, we define 
\[ \FuM^i_{[n]} \subset \FuM_{[n]}(\BC^2), \quad i\in \BZ_2,\]
as the closure of the configuration space of points of an axis, $C_{[n]}(\BC_{i})$, inside $\FuM_{[n]}(\BC^2)$. By the universality of blow-ups,  it is naturally isomorphic to the FM space of an affine line, 
 \[ \FuM^i_{[n]}= \FuM_{[n]}(\BC).\] 
By \cite[Proposition 3.4.3]{CGK}, taking the sublocus of $\FuM^i_{[n]}$ where all marked points of FM degenerations lie on bubbles over the origin $0\in \BC_{i}$, we obtain the moduli space of genus zero curves inside $\FuM_{[n]}(\BC^2)$, 
\[ \Mbar^i_{0,n+1} \subset  \FuM_{[n]}(\BC^2),\]
where, as before, the superscript $i \in \BZ_2$ indicates the axis of $\BC^2$, and the extra marking comes from the point of attachment  a genus-0 curve to $ \BC_i$. 
More generally, we can iterate this construction by consecutively attaching  $\Mbar^i_{0,n+1}$ and $\Mbar^{i+1}_{0,n+1}$ to each other, such that at each new attachment we change the axis. For example, let $N_1 \cup N_2=[n]$ be a partition, then  there is a locus 
\[\Mbar^0_{0,N_1 \cup \{p_\infty,p_0\}} \times \Mbar^1_{0,N_2\cup \{p_0\}} \subset \FuM_{[n]}(\BC^2), \]
whose generic point lie on a FM degeneration 
\[\mathrm{Bl}_{0} (\BC^2) \sqcup_{E_1=D_1} \mathrm{Bl}_{p_2}(\p^2)\sqcup_{E_2=D_2} \p^2,\]
 such that points marked by $N_1$ lie on $x$-axis of the first bubble $\mathrm{Bl}_p(\p^2)$, while the points marked by $N_2$ lie on the $y$-axis of the second bubble $\p^2$. Note that ``lying on a $x$-axis or $y$-axis of a bubble" is a notion invariant with respect to automorphisms of bubbles, as it depends on the intersection of the axis with the divisor at infinity. More precisely, these embeddings are given by the restriction of a gluing map, 
\[ \FuM_{N_1\cup \{p_0\}}(\BC^2) \times T_{N_2} \rightarrow \FuM_{[n]}(\BC^2), \]
 to the loci  $\Mbar^0_{0,N_1\cup \{p_\infty, p_0\}} \times \Mbar^1_{0,N_2\cup\{p_0\}}$ on the first and the second factor, respectively, where $T_{N_2}$ is the locus considered in the previous subsection. The gluing map above attaches exceptional divisors of blow-ups of the marked point $p_0$ to the distinguished divisor at infinity of trees of projective spaces in $T_{N_2}$ considered, for example, in \cite[Section 3.3]{CGK}. Note that this is not the same loci as the one obtained by restricting to same axis on both factors, $\Mbar^i_{0,N_1\cup \{p_\infty, p_0\}} \times \Mbar^i_{0,N_2\cup\{p_0\}}$, since the later  is just a boundary loci of   $\Mbar^i_{0,n+1}$. 

There is a convenient way to label such loci. It is provided by rooted-tree graphs with leaves.

\begin{defn} Let $\Gamma^{i}$ be a rooted-tree graph with $n$ leaves. We label the tree by $i\in \BZ_2$.  Let $V(\Gamma^i)$, $E(\Gamma^i)$ and $L(\Gamma^i)$ be the set of vertices, edges and leaves of $\Gamma^i$, respectively, such that $L(\Gamma^i)=[n]$; $E(v)$ and $L(v)$ are edges and leaves connected to a vertex $v \in V(\Gamma^i)$.   We denote the root vertex by $v_\infty$, the unique edge of a vertex $v$ directed to $v_\infty$ by $e_\infty(v)$, and  the number of vertices between $v_\infty$ and $v$, including $v$ itself, by $\ell(v)$.  We require that $\Gamma^i$ is stable: $\mathrm{val}(v_\infty)\geq 2$ and  $\mathrm{val}(v)\geq 3$ for all other vertices $v$. 
	
\end{defn} 

By using the loci  $\Mbar^i_{0,k}$, for each graph $\Gamma^i$, we obtain a locus 
\[\Mbar_{\Gamma^i} \subset \FuM_{[n]}(\BC^2),\]
which is constructed as follows.  Consider the gluing map
\begin{multline*}
  \FuM_{ L(v_\infty)\cup E(v_\infty)}(\BC^2) \times\prod_{v \in V(\Gamma^i) \setminus \{v_\infty\} }T_{L(v)\cup E(v)\setminus \{e_\infty(v)\}}
  \rightarrow \FuM_{{[n]}}(\BC^2),
  \end{multline*}
such that for each edge $e\in E(\Gamma^i)$, components corresponding to vertices that $e$ connects are attached to each other via the exceptional divisor associated to the blow-up of a point and the divisor at infinity; the latter we treat as a marking associated to $e_\infty(v)$. We then restrict the gluing map to

\[\Mbar^i_{0, L(v_\infty)\cup E(v_\infty) \cup \{p_\infty\}} \times\prod_{v \in V(\Gamma^i) \setminus \{v_\infty\} } \Mbar^{i+\ell(v)}_{0, L(v)\cup E(v)},\]
this is where the superscript $i\in \BZ_2$ in $\Gamma^i$ enters the construction.  The image of the gluing map applied to the locus above is $\Mbar_{\Gamma^i}$. 

The importance of these loci is in the fact that $T$-fixed loci of $\FuM_{[n]}(\BC^2)$ are built out of them. 
\begin{prop} Let $T=(\BC^*)^2$, we have
	\[\FuM_{[n]}(\BC^2)^G=\coprod_{\Gamma^0} \Mbar_{\Gamma^0} \sqcup \coprod_{\Gamma^1 } \Mbar_{\Gamma^1}.\]	
\end{prop}
\begin{proof} Once points are on the bubbles, there are $\BC^*$-scaling  automorphisms of bubbles.  Hence if points lie on an axis, then they are fixed by the torus $T$. For example, on $W=\mathrm{Bl}_0(\BC^2) \sqcup_{E=D}\p^2$, the torus $T$ scales $\p^2$ by \[(t_1,t_2) \cdot [x, y, z]= [t_1x, t_2y, z],\]
but since there are $\BC^*$-scaling automorphisms of $\p^2$, the fixed locus includes $n$-tuples of points with coordinates  $[x, 0, 1]$ or $[0,y,1]$, but not both. 
	
	 Conversely, if they do not lie on the axis, then they are moved by the torus. After some contemplation, one can see that all possible configurations of points lying on axes  are give by the loci $\Mbar_{\Gamma^i}$. Moreover, all of $\Mbar_{\Gamma^i}$ are distinct. In particular, none of $\Mbar_{\Gamma^i}$ is a sublocus of another $\Mbar_{\Gamma'^{i'}}$, because we alternate the $x$- and $y$-axis in the construction. Points lying on one axis cannot deform to points lying on another axis inside the $T$-fixed locus.   
	\end{proof}



\section{On $\psi$-classes} \label{comp}

\subsection{Comparison of $\psi$-classes} 

In higher dimensions, there are essentially three different definitions of $\psi$-classes. One is more natural for the proof of the wall-crossing formula, Section \ref{psiexc}. The other two are easier to think about or to do computations with. In this section, we show that all of them are equivalent.  

Recall that the boundary of $\FuM_{[n]}(X)$ is a simple normal crossing divisor, 
\[\partial \FuM_{[n]}(X)= \FuM_{[n]}(X)\setminus C_{[n]}(X)= \bigcup_{|I|\geq 2} D_I,\] 
where a general point of $D_I$ is a marked FM degeneration $(W, \underline{p})$, such that 
\[ W=\mathrm{Bl}_p(X) \sqcup_{D=E} \p^d, \]
and $p_i$ lies on $\p^d$, if $i \in I\subseteq [n]$. We refer to \cite{FM} for more details on these divisors. 
\begin{lemma} \label{Li} On $\FuM_{[n]}(X)$, we have the following identity, 
	
	\begin{equation*}
		\BL_i=(\pi_i^*T^\vee_X)(\sum_{i\in I} D_I).
	\end{equation*}
	
\end{lemma}

\begin{proof}
	
	Let $U= \FuM_{[n]}(X) \setminus D_{i,n}$, and $\mathrm{fg}_n \colon \FuM_{[n]}(X) \rightarrow \FuM_{[n-1]}(X)$ be the forgetful map,  then we have 
	\begin{itemize}
		\item $\mathrm{fg}^*_n (\BL_i)_{\mid U}=\BL_{i\mid U}$, 
		\item $\BL_i=T_X$ on $\FuM_{\{i\}}(X)=X$. 
	\end{itemize}
	Since $\FuM_{[n]}(X)$ is smooth, the above identification extends from $U$ to some morphism on the whole $\FuM_{[n]}(X)$ after a twist by a power of $\CO(D_{i,n})$,
	\[\phi \colon \mathrm{fg}^*_n (\BL_i)(kD_{i,n}) \rightarrow \BL_{i}.   \]
	Moreover, if it is an isomorphism away from a codimension $2$ locus for some $k$, then it is an isomorphism. We also have the following relation for the pullback of divisors $D_{I}$, 
	\[ \mathrm{fg}^*_{n}D_{I}= D_{I} +D_{I \cup \{n\} }.\] 
	  By the  induction on $n$, to prove the claim, it is therefore enough to show that  $\phi$ is an isomorphism away from a codimension $2$ locus for $k=1$. 
	We can check this in a formal neighbourhood around a point supported on $D_{i,n}$ away from the other boundary strata.

	By the simple normal crossing condition, the formal neighbourhood  of the pair $(\FuM_{[n]}(X), D_{i,n})$ away from the other boundary strata looks like
	\[\mathrm{Spf}(\BC[\![x_1, \dots,x_{nd} ]\!]), \quad \CO(-D_{i,n})=(x_1).\] 
	We therefore have to show that  the restriction of $\BL_i^\vee$  is equal\footnote{More precisely, it has to be isomorphic to $(x_1)^{\oplus d}$, such that away from $D_{i,n}$ the isomorphism is the identity.  } to $(x_1)^{\oplus d}$. To do so, we can further restrict to 
	\[\Delta=\mathrm{Spf}(\BC[\![x_1, \dots,x_{nd} ]\!]/(x_2,\dots, x_{nd})).\]

 Consider the blow-up $\mathrm{Bl}_{0}(\BC^{d} \times \Delta)$ at the origin of $\BC^{d} \times \Delta$.  Let 
	\[ \iota \colon \Delta \hookrightarrow \mathrm{Bl}_{0}(\BC^{d} \times \Delta), \quad \pi \colon \mathrm{Bl}_{0}(\BC^{d} \times \Delta) \rightarrow \Delta  \]
	be the total transform of $\{0\}\times \Delta$ and the natural projection, respectively. Then the restriction of $\BL^\vee_i$ to $\Delta$ is of the following form, 
	\[ \BL^\vee_{i|\Delta}=\iota^* T_\pi.\]
	By restricting to the strict transform of the hyperplane axis $\BC^{d-1} \times \Delta \hookrightarrow \mathrm{Bl}_{0}(\BC^{d} \times \Delta)$,
	we have a split exact sequence 
	\[ 0 \rightarrow  \iota^*T_{\pi'}  \rightarrow \iota^*T_\pi \rightarrow \iota^*\CN\rightarrow 0,\]
	where $\CN$ is the normal bundle of the strict transform of the hyperplane axis, and $\pi'\colon \mathrm{Bl}_{0}(\BC^{d-1}\times \Delta) \rightarrow \Delta$ is the natural projection. 
	 The term $\iota^*\CN$ can be easily determined,
	\[ \iota^*\CN=(x_1).\]
	On the other hand, for $d=1$, 
	\[\iota^*T_\pi=(x_1).\]
	By induction on $d$, we therefore obtain that 
	\[  \BL^\vee_i=(x_1)^{\oplus d},\]
	which implies the claim. 
	\end{proof}

\subsection{Exceptional divisors and $\psi$-classes} \label{psiexc}

Consider $\FuM_{[n]}(X)$ together with the universal FM degeneration and the universal sections, 
\[
\begin{tikzcd}[row sep=normal, column sep = normal]
	\CW  \arrow[r] &  \FuM_{[n]}(X), \arrow[l,bend right=25, swap, "s_i"] &\hspace{-0.8cm} i=1,\dots, n.
\end{tikzcd}
\]
Let $\CW(n)$ be the $\FuM_{[n]}(X)$-relative blow-up of sections $s_i$. Fibers of $\CW(n)$ are FM degenerations blown-up at $n$ distinct smooth points. Let 

\[ \CE_i \hookrightarrow\CW(n), \quad i=1,\dots , n\]
be the exceptional divisor associated to the blow-up at the section $s_i$. The relative normal bundle of the section $s_i$ is $\BL_i$. Hence by Lemma \ref{Li}, we have the identification, 
\[ \CE_i = \p(\BL^\vee_i)\cong  \p(\pi_i^*T_X),\]
which therefore induces  a map, 
\begin{equation} \label{ident2}
	\CE_i  \rightarrow \p(T_X),
\end{equation}
providing the identification of fibers of $\CE_i$ over $\FuM_{[n]}(X)$ with $\p(T_{X,x})$.  We denote the $\FuM_{[n]}(X)$-relative normal bundle of $\CE_i$ inside $\CW(n)$ by $\CN_{\CE_i}$, and the pullback of the relative tautological bundle $\CO_{\p(T_X)}(1)$ by the map (\ref{ident2})  from $\p(T_X)$ by $\CO_{\CE_i}(1)$.  We define the cotangent bundle associated to $\CE_i$ and the corresponding $\psi$-classes as follows, 
\[ \BL(\CE_i)^\vee:= p_{\CW(n)*}(\CN_{\CE_i} \otimes \CO_{\CE_i}(1)),  \quad \psi(\CE_i):=\mathrm{c}_1(\BL(\CE_i)).\]
\begin{cor}  \label{corpsi} We have, 
	\[ 	\BL(\CE_i)=\CO(\sum_{i\in I} D_I), \]
	which, in particular, implies that
	\[ \psi_i=\sum_{i\in I} D_I=\mathrm{c}_1(\BL(\CE_i)).\]
\end{cor}

\begin{proof} The first claim follows from the following identities,

\begin{itemize}
	\item $\mathrm{fg}^*_n (\BL(\CE_i))_{\mid U}=\BL(\CE_i)_{\mid U}$, 
	\item $\BL(\CE_i)=\CO$ on $\FuM_{\{i\}}(X)=X$,
\end{itemize}
and exactly the same arguments as in Lemma \ref{Li}. The second claim follows from the first, and Lemma \ref{Li}, 
\[ \psi_i=(\mathrm{c}_1(\BL_i)-\pi_i^*\mathrm{c}_1(X))/d=\sum_{i\in I}D_I.\]
\end{proof}


\subsection{Dilaton and string equations} \label{sectiondil}
Let 
\[\mathrm{fg}_n \colon \FuM_{[n]}(X) \rightarrow \FuM_{[n-1]}(X) \]
be the forgetful map. We have 
\begin{equation} \label{f0}
	\psi_i=\mathrm{fg}_n^*\psi_i+D_{i,n},
\end{equation}
By the construction of $\FuM_{[n]}(X)$ and its universal family $\CW$, \cite{FM}, 
\[ \FuM_{[n]}(X) \cong \mathrm{Bl}_{s_1, \dots, s_{n-1}}(\CW), \]
where $s_i$ are sections of the universal FM degeneration $\CW$ over $\FuM_{[n-1]}(X)$, we blow them up relative to $\FuM_{[n-1]}(X)$. The divisors $D_{i,n}$ are exactly the  exceptional divisors of the blow-ups. Since $\BL^\vee_i$ is the relative normal bundle of $s_i$, 
\begin{equation*} \label{bundary}
	D_{i,n}=\p(\BL_i^\vee).
\end{equation*}
We therefore obtain that
\begin{equation} \label{f2}
	D_{i,n}\cdot D_{i,n}=\mathrm{c}_1(\CO_{\p(\BL_i^\vee)}(-1)).
\end{equation}

On the other hand, by Lemma \ref{Li}, we have 
\[ \p(\BL_i^\vee)\cong \p(\pi_i^*T_X). \]
It is also clear that on $D_{i,n}$ the class  $\mathrm{fg}_n^*\psi_i$ is the pullback of $\psi_i$ from $\FuM_{[n-1]}(X)$ to $D_{i,n}$. Hence, by  Corollary \ref{corpsi}, (\ref{f0}) and (\ref{f2}), we obtain that 
\begin{equation} \label{f11}
	\psi_i \cdot D_{i,n}=-H:=\mathrm{c}_1(\CO_{\p(\pi_i^*T_X )}(-1)).
\end{equation}
We are now ready to prove the dilaton equation. 

\begin{lemma} \label{dilaton}
	The dilaton equation holds, 
	\[\langle \psi_1^{k_1} \gamma_1 \dots \psi_{n}^{k_n}\gamma_n \cdot \tilde{\psi}_{n+1}  \rangle^{\FuM}_{n+1} =(-1)^{d} (\mathrm{c}_{d}(X)-n) \cdot \langle \psi_1^{k_1} \gamma_1 \dots \psi^{k_n}\gamma_n   \rangle^{\FuM}_{n}  .\]	
\end{lemma}
\begin{proof} Firstly, by definition,   
	\[ 
	\tilde{\psi}_{n+1} =\sum^{d}_{\ell=0} (-1)^\ell \psi_{n+1}^{d-\ell}  \pi_{n+1}^*\mathrm{c}_\ell(X).
	\]
	The Chow group of $\p(\pi_{n+1}^*T_X)$ has the following relation,
	\[ 
	H^{d}+H^{d-1}\mathrm{c}_1(X)+\dots+\mathrm{c}_{d}(X)=0,
	\]
	which by (\ref{f11}) implies that
	\[ 
	\tilde{\psi}_{n+1}\cdot D_{i,n+1} = \left( \sum (-1)^{\ell}\psi_{n+1}^{d-\ell} \pi^*_{n+1}\mathrm{c}_\ell(X))\right) \cdot  D_{i,n+1}=0. 
	\]
	Hence by (\ref{f0}), we have
	\begin{align*}
		\langle \psi_1^{k_1} \gamma_1 \dots \psi_{n}^{k_n}\gamma_n \cdot \tilde{\psi}_{n+1} \rangle^{\FuM}_{n+1} 
		&=  \tilde{\psi}_{n+1}  \cdot  \mathrm{fg}^*_{n+1}  (\psi_1^{k_1} \gamma_1) \dots \mathrm{fg}^*_{n+1}  (\psi_{n}^{k_n}\gamma_n)\\
	 &=\mathrm{fg}_{n+1*}(\tilde{\psi}_{n+1}) \cdot    \psi_1^{k_1} \gamma_1 \dots  \psi_{n}^{k_n}\gamma_n.
	\end{align*}
	It remains to compute  $\mathrm{fg}_{n+1*}(\tilde{\psi}_{n+1}) =\mathrm{fg}_{n+1*}\left(\sum (-1)^{\ell}  \psi_{n+1}^{d-\ell} \pi^*_{n+1}\mathrm{c}_\ell(X)\right)$. This can be done on a smooth fiber, since $\FuM_{[n]}(X)$ is connected. A smooth fiber is a blow-up of $X$ in $n$ points, such that the restriction of $-\psi_{n+1}$ is the sum of the exceptional divisors. Hence 
	\begin{align*}
		\mathrm{fg}_{n+1*} \left(\sum  (-1)^\ell\psi_{n+1}^{d-\ell} \pi_{n+1}^*\mathrm{c}_\ell(X)\right)&=\mathrm{fg}_{n+1*} \psi_{n+1}^{d}+(-1)^{d}\mathrm{fg}_{n+1*} \pi^*_{n+1}\mathrm{c}_{d}(X) \\ &=(-1)^{d}\ (\mathrm{c}_{d}(X)-n), 
	\end{align*}
	which concludes the proof. 
\end{proof}

We also can prove a version of the string equation. 

\begin{lemma} \label{string}The string equation holds, 
	\begin{multline*}
		\langle \tilde{\psi}_1\psi_1^{k_1} \gamma_1 \dots  \tilde{\psi}_n\psi_{n}^{k_n}\gamma_n \cdot \mathbb{1}  \rangle^{\FuM}_{n+1} \\
		=(-1)^{d-1}\sum^n_{i=1} \langle \tilde{\psi}_1\psi_1^{k_1} \gamma_1 \dots \widehat{\tilde{\psi}_i} \psi^{k_i}\gamma_i \dots  \tilde{\psi}_n \psi^{k_n}\gamma_n   \rangle^{\FuM}_{n} . 
	\end{multline*}
\end{lemma}

\begin{proof} By the relations in the Chow group of $\p(\pi_{n+1}^*T_X)$, (\ref{f0}) and (\ref{f11}), we have 
	\begin{align*}  
	\tilde{\psi}_i\cdot \psi_i^{k} = \tilde{\psi}_i\cdot \mathrm{fg}_{n+1}^*\psi_i^{k},
	\end{align*}
and by (\ref{f0}) and (\ref{f2}),
	\begin{align*}
	\mathrm{fg}_{n+1*} \tilde{\psi}_i=\mathrm{fg}_{n+1*} \psi_i^{d}&= \mathrm{fg}_{n+1*}(\mathrm{fg}_{n+1}^*\psi_i+D_{i,n+1})^{d}\\
	&= \mathrm{fg}_{n+1*}D_{i,n+1}^{d}=(-1)^{d-1}\mathbb 1_{\FuM_{[n]}(X)}.
	\end{align*}
Also, by the relations on $\FuM_{[n+1]}(X)$,
	\[ 
	D_{i,n+1}\cdot D_{j,n+1}=0,\text{ if } i\neq j. 
	\]
Overall, 
\begin{multline*}
 \langle \tilde{\psi}_1\psi_1^{k_1} \gamma_1 \dots  \tilde{\psi}_n\psi_{n}^{k_n}\gamma_n \cdot \mathbb{1}  \rangle^{\FuM}_{n+1}
 = \prod^{n}_{i=1}  (\mathrm{fg}^*_{n+1}\tilde{\psi}_i+ \dots+ D^d_{i,n+1})\mathrm{fg}^*_{n+1}(\psi_i^{k_i} \gamma_i) \\
 =  \sum^n_{i=1} \mathrm{fg}_{n+1*}(D_{i,n+1}^d) \tilde{\psi}_1\psi_1^{k_1} \gamma_1 \dots \widehat{\tilde{\psi}_i} \psi_i^{k_i}\gamma_i \dots  \tilde{\psi}_n \psi_n^{k_n}\gamma_n   \\
 = (-1)^{d-1}\sum^n_{i=1} \langle \tilde{\psi}_1\psi_1^{k_1} \gamma_1 \dots \widehat{\tilde{\psi}_i} \psi_i^{k_i}\gamma_i \dots  \tilde{\psi}_n \psi_n^{k_n}\gamma_n   \rangle^{\FuM}_{n},
 \end{multline*}
which concludes the proof. 
\end{proof}

\section{Proof of Theorem} \label{secproof}
\subsection{Entanglement of Zhou}

In this section, we prove the main theorem - Theorem \ref{maintheorem}. The proof of Theorem \ref{maintheorem2} is exactly the same. The wall-crossing formula is given by the residue of the torus localisation formula on the master space.  In this subsection we introduce the main ingredient of the proof - entanglement of Zhou \cite[Section 2]{YZ}.  The master space is introduced in Section \ref{sectionmaster}, the fixed components are analysed in Section \ref{secfixed}-\ref{Secnor}. The wall-crossing formula is then proved in Section \ref{Secres}.

For simplicity we assume that $X$ is projective. We comment on how we should adjust the proof for the non-projective equivariant setting in Section \ref{Seceq}

\begin{defn}  A weighted marked FM degeneration $(W, \underline{p}, \underline{n})$ is a FM degeneration $W=\cup^k_{i=1} W_i$ together with a vector of positive integers (weights) $\underline{n}=(n_1, \dots, n_k) \in \BZ^k_{>0}$ attached to each irreducible component $W_i$ of $W$. The total weight of $W$ is defined to be $n=\sum^{k}_{i=1} n_i$. We will often drop $\underline{n}$ from the notation. 
	
	A family of weighted FM degenerations over a base scheme $B$ is a family of FM degenerations $(\CW, \underline{p}) \rightarrow B$ with weights attached to each geometric point  over $B$, such that \'etale locally they are given by multiplicities of (possibly overlapping) sections $\{s_1, \dots, s_k\}$ of $\CW$ at each irreducible component of geometric fibers. We require sections to be contained in the smooth locus of fibers of $\CW$. 

	\end{defn}
	We define 
\[\CF\CM_{[m]}(X,n) \colon (Sch/\BC)^\circ \rightarrow Grpd \]
to be the moduli space of weighted marked FM degenerations of total weight $n$, such the weight of end components is at most $n_0 
\in \BZ$, which we do not indicate in the notation, as it will be clear from the context. Here, $Grpd$ denotes the $2$-category of groupoids.   It is not difficult to see, using the algebraicity of $\CF\CM_{[m]}(X)$, that $\CF\CM_{[m]}(X,n)$ is an algebraic stack.

 Let \[\CZ_k \subset \CF\CM_{[m]}(X,n)\] be a reduced closed substack which parametrizes weighted marked FM degenerations with at least $k$ end components of weight $n_0$. Since $\CZ_1\subset \CF\CM_{[m]}(X,n)$ is a simple normal crossing divisor, in formal neighbourhoods,  $\CZ_k$ is the  intersection of local branches of $\CZ_1$.  

Let $h:=\floor{n/n_0}$ be the maximum number of weight $n_0$ end components of FM degenerations in $\CF\CM_{[m]}(X,n)$, then $\CZ_h$ is the deepest stratum, and therefore is smooth.   Set 
\[\CM_h:=\CF\CM_{[m]}(X,n).\] Let 
\[\CM_{h-1} \rightarrow \CM_h\]
be the blow-up of $\CM_h$ at $\CZ_h$. For $i\in \{0,1,\dots, h-1\}$, we define $\CM_{i-1}$ to be the blow-up of $\CM_{i}$ at the proper transform of $\CZ_i$ in $\CM_{i}$,  denoted by $\CZ_{(i)}$, which is smooth.  We then set 
\[\widetilde{\CF\CM}_{[m]}(X,n):=\CM_0,\]
this is the space of weighted marked FM degenerations with \textit{entangled} end components, its purpose is to remove $\BC^*_z$-scaling automorphisms of end components. 
\begin{defn} 
	Given a  FM degeneration $(W,\underline{p}) \in \CF\CM_{[m]}(X,n)(\BC)$, an \textit{entanglement} of $(W,\underline{p})$ is a point in the fiber of $\widetilde{\CF\CM}_{[m]}(X,n) \rightarrow \CF\CM_{[m]}(X,n)$ over $(W,\underline{p})$. We denote an entanglement of $(W,\underline{p})$ by $( W, \underline{p}, e)$.
\end{defn}

Let $\{P_1, \dots, P_\ell\}$ be a set of the end components of $(W,\underline{p})$ of weight $n_0$. At $ (W,\underline{p}) \in  \CF\CM_{[m]}(X,n)$,  the divisor $\CZ_1$ has local branches $\CH_j$,  where an end component $P_j$ remains intact  (it does not smooth out). Define
\[k:= \mathrm{min}\{ i \mid \text{the image of } ( W, \underline{p},e) \text{ in } \CM_i \text{ lies in } \CZ_{(i)}   \}.\]
The image of the entangled FM degeneration $( W,\underline{p},e)$ in $\CM_i$ lies on the intersection of $k$ branches $\{\CH_{j_1}, \dots, \CH_{j_k}\}$. The corresponding set of end  components $\{P_{j_1}, \dots P_{j_k}\}$ are called \textit{entangled.} 

\begin{example} The simplest instance of the pair $(\CF\CM_{[m]}(X,n), \CZ_1)$ is given by $(\BA^2, Z=\{xy=0\})$ in a  smooth chart. The procedure described above tells us to blow $\BA^2$ at $0 \in \BA^2$. The proper transform $\widetilde{Z}$ of $Z$ is given by two disjoint lines $\BA^1$ intersecting the exceptional divisor $E$ at two distinct points. The entanglement of a point $w \in \BA^2$ is a point $\tilde{w} \in \mathrm{Bl}_0 (\BA^2)$ lying over it. It has two entangled tails, if $\tilde{w} \in E \setminus \tilde Z\cap E$, otherwise it has one entangled tail. 
	\end{example}

\begin{defn} We define the divisor
	\[\CY_i \subset \widetilde{\CF\CM}_{[m]}(X,n)\]
	to the closure of the locus of FM degenerations with exactly $i+1$ entangled end components. 
\end{defn}

\begin{defn} We define the \textit{calibration bundle} on $\widetilde{\CF\CM}_{[m]}(X,n)$, 
	\[ \BM_{\widetilde{\CF\CM}_{[m]}(X,n)}:=\CO(-\CZ_{(1)}).  \]
	We then define the moduli space of FM degenerations with \textit{calibrated} end components,
	\[\BM\widetilde{\CF\CM}_{[m]}(X,n):=\p(\BM_{\widetilde{\CF\CM}_{[m]}(X,n)} \oplus \CO_{\widetilde{\CF\CM}_{[m]}(X,n)}).\]
\end{defn}
A $B$-valued point of $M\widetilde{\CF\CM}_{[m]}(X,n)$ is given by 
\[ ( \CW,\underline{p},e,\CL,v_1,v_2),\]
where 
\begin{itemize}
	\item $( \CW,\underline{p},e) \in \widetilde{\CF\CM}_{[m]}(X,d)(B)$ is a FM degeneration with entanglement,
	\item  $\CL$ is a line bundle on $B$, such that $v_1 \in H^0(B, \BM_B\otimes \CL)$ and $v_2 \in H^0(B,  \CL)$ are sections with no common zeros.
\end{itemize} 
We refer to the data above as a \textit{calibration} of $(\CW,\underline{p})$. 

\subsection{Master space} \label{sectionmaster}

\begin{defn} \label{relative2} Given a wall $\epsilon_0=1/n_0 \in  \BR_{>0}$. A triple $(W, Z, \underline{p})$ is $\epsilon_0$-\textit{semi-weighted}, if 
	\begin{enumerate}
			
				\item $\underline{p} \cap Z=\emptyset$,
				\item $\underline{p} \subset W^{\mathrm{sm}}$ and $Z\subset W^{\mathrm{sm}}$, 
				\item for all $x \in W$, $\ell_x(Z)\leq 1/\epsilon$,
				\item  for all end components  $\p^{d} \subset W$, such that $\underline{p} \cap \p^{d}=\emptyset$,   $\ell(Z_{|\p^{d}})\geq 1/\epsilon$, 
				\item the group $\{g \in \Aut(W, \underline{p}) \mid g^*Z=Z\}$ is finite on ruled components of $W$.
	\end{enumerate}
	
\end{defn}

The difference between Definition \ref{defnweight} and the definition above is that we have a non-strict inequality in the condition (4), $\ell(Z_{|\p^{d}})\geq 1/\epsilon$, and we allow automorphisms on end components in the condition (5).  

Let  
	\[ \CF\CM^{\epsilon_0}_{n,[m]}(X) \colon (Sch/\BC)^\circ \rightarrow Grpd\] 
	be the moduli space of $\epsilon_0$-\textit{semi-weighted} triples $(W,Z,\underline{p} )$. It is an open substack of the relative Hilbert scheme $\Hilb_n(\CW/\CF\CM_{[m]}(X))$, hence is quasi-separated and algebraic. It admits a projection, 
	\begin{align*}
	\CF\CM^{\epsilon_0}_{n,[m]}(X) & \rightarrow \CF\CM_{[m]}(X,n), \\
	(W, Z,\underline{p} ) & \mapsto (W, \underline{p}, \ell(Z) ).
	\end{align*}
We define a moduli space of \textit{$\epsilon_0$-semi-unramified with calibrated end components} of length $n_0$ as the fibred product, 
\[\BM\CF\CM^{\epsilon_0}_{n,[m]}(X):=\CF\CM^{\epsilon_0}_{n,[m]}(X)\times_{\CF\CM_{[m]}(X,n)} \BM\widetilde{\CF\CM}_{[m]}(X,n). \]
A $B$-valued point of this moduli space is given by
\[\xi=(\eta,\lambda)=( \CW, Z, \underline{p},  e, \CL, v_{1}, v_{2}),\]
where $\eta$ is the underlying $\epsilon_0$-semi-weighted triple $(\CW,Z,\underline{p})$, and $\lambda$ is a calibration of $\CW$, 
\[\eta=( \CW,Z, \underline{p}) \ \text{and} \ \lambda=(e, \CL, v_1,v_2).\]

\begin{defn} Given a triple  $(W, Z, \underline{p})$. An end component $P \subset  (W,\underline{p})$ is \textit{constant}, if the support of $Z$ on $P$ consists of one point, or, equivalently, $\Aut_{P}(P,Z_{\mid P})=\BC^*_z$. 
\end{defn}
\begin{defn} A family of $\epsilon_0$-semi-weighted triples $(W, Z, \underline{p})$ with  calibrated end components,
	\[( W,Z, \underline{p}, e, \CL, v_{1}, v_{2}),\]
	is $\epsilon_0$-weighted, if 
	\begin{enumerate}
		\item all constant end component are entangled,
		\item if  $ W$ has an end component $P$ of weight $n_0$, then end components of $W$ contain all points $x$, such that $\ell_x(Z)=n_0$,
		\item if $v_1=0$, then $(W,Z, \underline{p})$ is $\epsilon_+$-weighted,
		\item if $v_2=0$, then $(W,Z,\underline{p})$ is $\epsilon_-$-weighted. 
	\end{enumerate} 
\end{defn}
\begin{defn}We define the \textit{master space},
	\[\BM \FuM^{\epsilon_0}_{n,[m]}(X) \subset \CF\CM^{\epsilon_0}_{n,[m]}(X) ,\]
	to be  the moduli space of  $\epsilon_0$-weighted triples $(W,Z, \underline{p})$ with calibrated end components. 
\end{defn}

\begin{thm} \label{Masproper} The moduli space $\BM \FuM^{\epsilon_0}_{n,[m]}(X)$ is a proper Deligne--Mumford stack. If $d\leq 3$, it has a perfect obstruction theory, given by the complex 
	\[R\CHom_{p_\CW}(\CI,\CI)_0, \]
where $\CI$ is the universal ideal and $p_\CW$ is the universal projection from the universal FM degeneration. 	
\end{thm}  

\textit{Proof.} The proof of the properness is exactly the same as in the case of maps in \cite[Section 6.4]{NuG}, which is a generalisation  \cite[Section 5]{YZ} to higher dimensions. For a given generic fiber, we classify all possible $\epsilon_0$-semi-weighted triples $(\CW,\CZ, \underline{p})$  with calibrated end components over the spectrum of a discrete valuation ring $\Delta$. All of them are related  by blow-ups and blow-downs of the central fiber of the FM degeneration $\CW$, while all other gadgets easily extend, e.g.\ $\CZ$ extends to any blow-up by the properness of relative Hilbert schemes, so does the calibration $\lambda$. The classification then depends of the singularity of end components in the central fiber that were introduced in the process of the blow-up, just like in \cite[Lemma 6.11, 6.12]{NuG}. Among all these $\epsilon_0$-semi-weighted triples, there is a unique one which is  $\epsilon_0$-weighted by the arguments from \cite[Section 5]{YZ}. 

The existence of the perfect obstruction theory follows from standard arguments applies to sheaves on moving varieties of dimension not greater than three. 
  \qed 

\subsection{Fixed loci} \label{secfixed}
There is a natural $\BC^*_z$-action on the space $\BM \FuM^{\epsilon_0}_{n,[m]}(X)$ given by  
\[z \cdot (W,Z,\underline{p}, e, \CL, v_1,v_2)= (W,Z,\underline{p}, e, \CL, z \cdot v_1,v_2), \quad z \in \BC^*_z.\] 
By arguments from \cite[Section 6]{YZ}, an $\epsilon_0$-weighted triple $(W,Z,\underline{p})$ with calibrated end components is $\BC^*_z$-fixed, if and only if it is one of the following: 
\begin{enumerate}
	\item $v_1=0$ and $(W,Z,\underline{p})$ is a $\epsilon_+$-weighted,
	\item $v_2=0$ and $(W,Z,\underline{p})$ is $\epsilon_-$-weighted,
	\item $v_1\neq 0$ and $v_2\neq0 $, and all weight-$n_0$ end components  of $(W,Z,\underline{p})$ are entangled strictly constant end components. 
\end{enumerate}
The $\BC^*_z$-fixed locus decomposes according to the types of triples above, 
\begin{equation} \label{fixed}
\BM \FuM^{\epsilon_0}_{n,[m]}(X)^{\BC^*_z}=F_- \sqcup F_+ \sqcup \coprod_{k\geq 1} F_{k},
\end{equation} 
the exact meaning of the components in the above decomposition is explained below. 
\subsubsection{$F_-$} This component admits the following description,  
\[
F_-= \FuM^{\epsilon_-}_{n,[m]}(X), \quad
N_{F_-}^{\mathrm{vir}}=\BM^{\vee}_-,\]
where $\BM^{\vee}_-$ is the dual of the calibration bundle $ \BM_{\widetilde{\CF\CM}_{[m]}(X,n)}$ on $\FuM^{\epsilon_-}_{n,[m]}(X)$ with a trivial $\BC^*_z$-action of weight $-1$. The obstruction theories also match, therefore 
\[[F_-]^{\mathrm{vir}}=[ \FuM^{\epsilon_-}_{n,[m]}(X)]^{\mathrm{vir}},\]
with respect to the identification above. 
\subsubsection{$F_+$} We define
\[\widetilde{\FuM}^{\epsilon_+}_{n,[m]}(X):=  \FuM^{\epsilon_+}_{n,[m]}(X) \times_{\CF\CM_{[m]}(X,n)} \widetilde{\CF\CM}_{[m]}(X,n),\] 
then
\[
F_+=\widetilde{\FuM}^{\epsilon_+}_{n,[m]}(X),\quad
N_{F_+}^{\mathrm{vir}}=\BM_+,
\]
where $\BM_+$ is the calibration bundle  $\BM_{\widetilde{\CF\CM}_{[m]}(X,n)}$ on $\widetilde{\FuM}^{\epsilon_+}_{n,[m]}(X)$ with a trivial $\BC^*_z$-action of weight $1$. The obstruction theories also match, and 
\[\pi_*[\widetilde{\FuM}^{\epsilon_+}_{n,[m]}(X)]^{\mathrm{vir}}=[ \FuM^{\epsilon_+}_{n,[m]}(X)]^{\mathrm{vir}},\]
where \[\pi \colon \widetilde{\FuM}^{\epsilon_+}_{n,[m]}(X) \rightarrow  \FuM^{\epsilon_+}_{n,[m]}(X) \] 
is the natural projection.
\subsubsection{$F_{k}$} \label{componentswall} These are the \textit{wall-crossing components}, which will account for the difference between $F_-$ and $F_+$.  They parametrize the third type of the fixed points,

\begin{equation*}	
	F_{k}:=	\left\{\xi \  \Bigl\rvert \ \arraycolsep=0.1pt\def\arraystretch{1} \begin{array}{c} \xi \text{ has exactly } k \text{ constant entangled } \\[.001cm] \text{end components of weight $n_0$}  \end{array} \right\}.
\end{equation*}
We now give a more explicit expression for $F_k$. 

\subsection{Wall-crossing components $F_k$ }
Let 
\[ G_X\subset \Aut_X(\p (T_X \oplus \BC))\]
be the group $X$-scheme of $X$-relative automorphisms of $\p (T_X \oplus \BC)$ which fix pointwise the relative divisor at infinity $\p (T X) \subset \p (T_X \oplus \BC)$. We define the \textit{moduli stack of end components}, 
\[\CB:=BG_X = [X/G_X],\]
which carries a universal family $\CP:=[\p (T_X \oplus \BC)/G_X]$  with the universal divisor $\CD:=[\p (T_X)/G_X]$ at infinity, 
\[\CD\hookrightarrow  \CP \rightarrow \CB. \]
Since the group $G_X$ fixes $\p(T_X) \subset \p (T_X \oplus \BC)$ pointwise, there exists a map 
\begin{equation} \label{ident1}
\CD \rightarrow \p(T_{X}), 
\end{equation}
which induces the identification of fibers of $\CD$ over $\CB$ with $\p(T_{X,x})$. 
\\

 Let us denote 
\[n':=n-kn_0.\]
As in Section \ref{psiexc}, we consider $\widetilde{\CF\CM}_{[m+k]}(X,n')$ with the universal FM degeneration together the universal sections, 
\[
\begin{tikzcd}[row sep=normal, column sep = normal]
	\CW  \arrow[r] &  \widetilde{\CF\CM}_{[m+k]}(X,n'), \arrow[l,bend right=25, swap, "s_{m+\ell}"]  &\hspace{-0.8cm} \ell=1,\dots, k.
\end{tikzcd}
\]
Let $\CW(k)$ be the $\widetilde{\CF\CM}_{[m+k]}(X,n')$-relative blow-up of sections $s_{m+\ell}$. Fibers of $\CW(k)$ are FM degenerations blown-up at $k$ distinct smooth points. Let 

\[ \CE_{m+\ell} \hookrightarrow \CW(k), \quad \ell=1,\dots , k\]
be the exceptional divisor associated to the blow-up at the section $s_{m+\ell}$. By Lemma \ref{Li}, it admits a map, 
\begin{equation} \label{ident3}
\CE_{m+\ell}  \rightarrow \p(T_X),
\end{equation}
which induces the identification of fibers of $\CE_{m+\ell}$ over $\widetilde{\CF\CM}_{[m+k]}(X,n')$ with $\p(T_{X,x})$. 

By using (\ref{ident1}) and (\ref{ident3}) to glue $\CW(k)$ with end components $\CP$ at the exceptional divisor and the divisor at infinity, respectively, we obtain the gluing map, 

\[\widetilde{\mathrm{gl}}_k \colon \widetilde{\CF\CM}_{[m+k]}(X,n')\times_{X^k} \CB^k \rightarrow \widetilde{\CF\CM}_{[m]}(X,n).\]

Finally, we are ready to give an explicit expression for wall-crossing components $F_k$ using $\widetilde{\mathrm{gl}}_k$.  Let
\[Y \rightarrow \widetilde{\CF\CM}^{\epsilon_+}_{[m+k]}(X,n')\] be the stack of $k$-roots of $\BM_{\widetilde{\CF\CM}_{[m+k]}(X,n')}$. By exactly the same arguments as in \cite[Proposition 6.7]{NuG}, we have a map 
\begin{equation}\label{themap}
	\widetilde{\mathrm{gl}}_k^*F_{k}  \rightarrow Y \times_{X^k} \prod_{\ell=1}^{k} V_{n_0},
\end{equation}
which is in fact an isomorphism. 

\begin{prop} \label{isomophism0} The map (\ref{themap}) is an isomorphism.
	\end{prop}

\begin{proof} Similar to the analysis of \cite[Section 6.7]{NuG} and \cite[Lemma 6.5.5]{YZ}. 
	\end{proof}

\subsection{The virtual normal complex of $F_k$} \label{Secnor}

We denote the $\CB$-relative normal bundle of $\CD$ inside $\CP$ by $\CN_{\CD}$, and the pullback of the relative tautological bundle $\CO_{\p(T_X)}(1)$ by (\ref{ident1})  from $\p(T_X)$ by $\CO_{\CD}(1)$.  We define the cotangent bundle associated to $\CD$ and the associated  $\psi$-classes as follows, 
\[ \BL(\CD)^\vee:= p_{\CP*}(\CN_{\CD} \otimes \CO_{\CD}(-1)),  \quad \psi(\CD):=\mathrm{c}_1(\BL(\CD)).\]

Given a FM degeneration $(W,\underline{p}) \in \widetilde{\CF\CM}_{[m]}(X,n) $ with end components $P_1, \dots, P_k$. Let $D_\ell$ and $E_\ell$ be the divisor at infinity and the exceptional divisor associated to an end component $P_\ell$. By \cite{Frie}, with respect to the map $\widetilde{\mathrm{gl}}_k$, the space $\Theta_\ell\cong \BC$ of first-order smoothings of  an end component $P_\ell \subset  (W,\underline{p})$ admits a natural identification
\[ \Theta_\ell \cong H^0(D_{\ell}, \CN_{D_{\ell}}\otimes \CN_{E_\ell})\cong (\BL(\CD_\ell)\otimes \BL(\CE_\ell))_{|(W,\underline{p})}^\vee.\]  Hence the fiber of the calibration bundle $ \BM_{\widetilde{\CF\CM}_{[m]}(X,n)}$ over $(W,\underline{p})$ is naturally isomorphic to
\begin{equation} \label{theta}
	(\Theta_1 \otimes \dots \otimes \Theta_k)^\vee.
\end{equation}
The appearance of line bundles $\CL(\CD_\ell)$ and $\CL(\CE_\ell)$  in the expression below  is therefore due to the splitting of simple normal crossing singularities of FM degenerations at end components.

\begin{prop} \label{isomophism} With respect to the isomorphism (\ref{themap}), we have 
	\begin{align*} [\widetilde{\mathrm{gl}}_k^*F_{k}]^{\mathrm{vir}} =& [Y]^{\mathrm{vir}} \times_{X^k} \prod_{\ell=1}^{k} [V_{n_0}]^{\mathrm{vir}},  \\	\frac{1}{e_{\BC^*_z}(\widetilde{\mathrm{gl}}_k^* N_{F_{k}}^{\mathrm{vir}})}&=\frac{\prod^k_{\ell=1}(\sum_j \mathrm{c}_j(X)(z/k+\psi(\CD_\ell))^j) )}{-z/k-\psi(\CD_1)-\psi(\CE_{m+1})-\sum^{\infty}_{i=k}\CY_i} \cdot \prod^k_{\ell=1} \CI_{n_0}(z/k+\psi(\CD_\ell)).
	\end{align*}
\end{prop}

\begin{proof} See \cite[Proposition 6.19]{NuG} and \cite[Lemma 6.5.6]{YZ}.  The difference is that the obstruction theory of maps is replaced by the obstruction theory of sheaves. The contribution from $\BM\widetilde{\CF\CM}_{[m]}(X,n)$ is exactly the same as in \cite[Proposition 6.19]{NuG}. 
	\end{proof}

\subsection{Analysis of the residue} \label{Secres}
Firstly, by Lemma \ref{Li},  we have 
\[ \psi(\CE_j)=\psi_j. \]
We also introduce the following notation,
\[A:=\prod^N_{i=1} \tau_{k_i}(\gamma_i), \quad A_{N_j}:=\prod_{i\in N_\ell} \tau_{k_i}(\gamma_i), \quad B=\prod^{j=m}_{j=1} \psi_j^{k'_j}\gamma'_j,\]
where $N_\ell \subseteq \{1,\dots, N\}$. 

By Theorem \ref{Masproper}, the master space $\BM\FuM^{\epsilon_0}_{n,[m]}(X)$ is proper. Hence by the virtual localization formula, the sum of residues at $z=0$ of localised classes associated to its $\BC^*_z$-fixed components (\ref{fixed}) is zero, 

\begin{multline} \label{therelation}
	\mathrm{Res}_z\left( \frac{[F_-]^{\mathrm{vir}}}{e_{\BC^*_z}(N^{\mathrm{vir}}_{F_-})}+\frac{[F_+]^{\mathrm{vir}}}{e_{\BC^*_z}(N^{\mathrm{vir}}_{F_+})}+\sum_{k\geq 1}\frac{[F_{k}]^{\mathrm{vir}}}{e_{\BC^*_z}(N^{\mathrm{vir}}_{F_{k}})} \right) \\
	=\mathrm{Res}_z([\BM\FuM^{\epsilon_0}_{n,[m]}(X)]^{\mathrm{vir}})=0,
\end{multline}
where $\mathrm{Res}_z$ is the $z=0$ residue, i.e.\ the coefficient of the term $1/z$ of a Laurent series. We interpret the equality above as a relation of cycles on any common space to which all terms admit a projection, like a point.  Recall that we expand all rational functions in $z$ in the range $|z|>1$. 

By Section \ref{secfixed}, the virtual tangent normal complexes of $F_-$ and $F_+$ are line bundles of weight $-1$ and $1$, respectively. Hence,
\[ \mathrm{Res}_z\left( \frac{[F_{-}]^{\mathrm{vir}}}{e_{\BC^*_z}(N^{\mathrm{vir}}_{F_{-}})}\right)=-[F_{-}]^{\mathrm{vir}}, \quad \mathrm{Res}_z\left( \frac{[F_{+}]^{\mathrm{vir}}}{e_{\BC^*_z}(N^{\mathrm{vir}}_{F_{+}})}\right)=[F_{+}]^{\mathrm{vir}}.\]
By inserting the class $A\cdot B$ and pushing forward  to the point the residue relation (\ref{therelation}), we therefore obtain 
\begin{equation} \label{theformula}
	\int_{[ \FuM^{\epsilon_+}_{n,[m]}(X)]^{\mathrm{vir}}}A \cdot B- \int_{[ \FuM^{\epsilon_-}_{n,[m]}(X)]^{\mathrm{vir}}} A\cdot B=- \sum_{k\geq 1}\mathrm{Res}_z \left(\int_{\frac{[F_{k}]^{\mathrm{vir}}}{e_{\BC^*_z}(N^{\mathrm{vir}}_{F_{k}})}}A\cdot B \right).
\end{equation}

We will now analyse the contribution of the wall-crossing components $F_{k}$. Using the second line of Proposition \ref{isomophism} and the fact that the gluing map $\widetilde{\mathrm{gl}}_k$ is of degree $k!$, we obtain that the residue associated to $F_{k}$ is equal to 
\begin{multline} \label{f1}
	\int_{[\widetilde{\mathrm{gl}}_k^*F_{k}]^{\mathrm{vir}} } \frac{A\cdot B }{k!}\cdot \mathrm{Res}_{z}\bigg(  \frac{ \prod^k_{\ell=1}(\sum_j \mathrm{c}_j(X)(z/k+\psi(\CD_\ell))^j) )}{-z/k-\psi(\CD_1)-\psi_{m+1}-\sum^{\infty}_{i=k}\CY_i} \\ 
	\cdot \prod^k_{\ell=1} \CI_{n_0}(z/k+\psi(\CD_\ell))\bigg). 
\end{multline}
By \cite[Lemma 2.7.3]{YZ}, the pullback of $\CY_i$ to  $\widetilde{\mathrm{gl}}_k^*F_{k}$ is equal to the pullback of the boundary divisor $\CY'_{i-k}$ from $ \widetilde{\CF\CM}_{ [m+k]}(X,n')$, which is defined in the same way but on $\widetilde{\CF\CM}_{[m+k]}(X, n')$. We also make the change of variables, which affects the residue only by a factor of $k$,  
\[ z \rightarrow k(z- \psi(\CD_1)-\psi_{m+1})=\dots=k(z- \psi(\CD_k)-\psi_{m+k}).\]
Overall, we obtain that (\ref{f1}) is equal to
\begin{multline} \label{f4}
	\int_{[\widetilde{\mathrm{gl}}_k^*F_{k}]^{\mathrm{vir}} } \frac{A\cdot B}{(k-1)!}\cdot \mathrm{Res}_{z}\bigg(  \frac{\prod^k_{\ell=1}(\sum_j \mathrm{c}_j(X)(z-\psi_{m+\ell})^j) )}{-z-\sum^{\infty}_{i=0}\CY'_i} \\
	\cdot \prod^k_{\ell=1} \CI_{n_0}(z-\psi_{m+\ell})\bigg).
\end{multline}

We now use the identification, 
\[ {\mathrm{gl}}_k^*F_{k}\cong Y \times_{X^k} \prod_{\ell=1}^{k} V_{n_0}. \]
On a FM degeneration $\CW=\CW_1 \cup \dots\cup \CW_{k+1}$,  a the structure sheaf of  zero-dimensional subscheme $\CO_Z$ splits as $\CO_{Z_1}\oplus \dots \oplus \CO_{Z_{k+1}}$, as long as it does not intersect the singular locus.  Wo conclude that over  $Y \times_{X^k} \prod_{\ell=1}^{k} V_{n_0}$, we have 
\begin{equation} \label{split}
 \tau_{k_i}(\gamma_i) =\tau_{k_i}(\gamma_i)_{|Y}\boxtimes \mathbb{1} \boxtimes \dots \boxtimes \mathbb{1}+\sum^{k}_{\ell=1} \mathbb{1} \boxtimes ... \boxtimes \tau_{k_i}(\gamma_i)_{|V_{n_0}} \boxtimes ... \boxtimes \mathbb{1}. 
 \end{equation}
Hence   the insertion $A$ splits among different components into insertions $A_{N_\ell}$,
\[A=\sum_{\underline{[N]}} A_{N'} \cdot A_{N_1} \dots  A_{N_k},\]
such that $A_{N'}$ is the insertions over $Y$ and $A_{N_\ell}$ are over the factors $V_{n_0}$. We allow $N'$ and $N_{\ell}$ to be empty. On the other hand, the insertion $B$ always remains on the component $Y$. 

 Applying the projection formula with respect to  $\widetilde{\mathrm{gl}}_k^*F_{k} \rightarrow \FuM^{\epsilon_+}_{n',[m+k]}(X)$ and  remembering that $Y \rightarrow  \FuM^{\epsilon_+}_{n',[m+k]}(X)$ is of degree  $1/k$, the formula (\ref{f4}) therefore becomes 

	\begin{equation*}
	\sum_{\underline{[N]}}	\int_{\FuM^{\epsilon_+}_{n',[m+k]}(X)} \frac{A_{N'}\cdot B}{k!} \cdot 
		\mathrm{Res}_z\left( \frac{\prod^{k}_{\ell=1}\pi^*_{m+\ell}	I_{n_0}(z-\psi_{m+\ell},A_{N_\ell})}{-z-\sum^\infty_{i=0} \CY'_i} \right). 
	\end{equation*} 
Let us introduce a shortened notation for the residue above, 
\[
C=\mathrm{Res}_z\left( \frac{\prod^{k}_{\ell=1}\pi^*_{m+\ell}	I_{n_0}(z-\psi_{m+\ell},A_{N_\ell})}{-z-\sum^\infty_{i=0} \CY'_i} \right).
\] 
We need to analyse the contributions of the divisors $\CY'_i$. We have 
\[ \frac{1}{-z-\sum^\infty_{i=0} \CY'_i}= \frac{-1}{z}+ \sum_{s\geq1}\sum_{r\geq 1}z^{-s}\CY'_{r-1}\cdot ( \sum_{i\geq 0} \CY'_i)^{s-1}. \]
For a Laurent polynomial $f(z)=\sum a_iz^i$ we define $[f(z)]_i=a_i$. Hence by the same argument as in \cite[Lemma 7.3.1]{YZ}, we obtain 
\begin{multline}
 \int_{\FuM^{\epsilon_+}_{n',[m+k]}(X)  } A_{N'}\cdot B\cdot C \cdot \CY'_{r-1}\cdot ( \sum_{i\geq 0} \CY'_i)^{s-1}\\
 =\sum_r\sum_{\underline{[N']}}\sum_{\underline{j}} \int_{\FuM^{\epsilon_+}_{n'',[m+k+r]}(X) }\frac{(-1)^{s-r} A_{N''}\cdot B\cdot C}{r!} \\
 \cdot \prod_{a=1}^{a=r}[\pi^*_{m+k+a}I_{n_0}(z-\psi,A_{N'_{a}} )]_{-j_a-1},
 \end{multline} 
where $\underline{j}=(j_1,\dots j_r)$ is an $r$-tuple of non-negative integers such that $j_1+...+j_r=s-r$. 
The difference with \cite{YZ} is that the sheaf-theoretic insertions  $A_{N'}$ have to be taken into account, by (\ref{split}) they split among the components that appear in the formula above. 

Plugging in this expression into the right-hand side of (\ref{theformula}), and relabelling the marked points, we obtain 
\[ \sum_{k\geq 1} \sum_{\underline{[N]}} \sum^{k-1}_{r=0} \sum_{\underline{b}} \frac{(-1)^r}{r!(k-r)!} \int_{\FuM^{\epsilon_+}_{n',[m+k]}(X) } A_{N'} \cdot  \prod_{\ell=1}^{k}[\pi^*_{n+\ell}I_{n_0}(z-\psi,A_{N_{\ell}} )]_{b_\ell}, \]
where we sum over integers $k$, partitions of the set $\{1,\dots, N\}$ and $\underline{b}$ runs through $k$-tuples of integers such that $b_1 + ...+ b_k=0$, and $b_{k-r+1},\dots, b_k<0$. By noticing that the terms with $b_\ell<0$ cancel out due to the sign, we obtain that the right-hand side of (\ref{theformula}) simplifies to
\[ \sum_{k\geq 1}\sum_{\underline{[N]}}  \frac{1}{k!} \int_{\FuM^{\epsilon_+}_{n',[m+k]}(X) } A_{N'} \cdot  \prod_{\ell=1}^{k}[\pi^*_{n+\ell}I(z-\psi,A_{N_{\ell}} )]_{0}, \]
which is exactly the wall-crossing formula, since 
\[[\pi^*_{n+\ell}I(z-\psi,A_{N_{\ell}} )]_{0}=\pi^*_{n+\ell}I(-\psi,A_{N_{\ell}} ),\]
where we use the convention of Theorem \ref{maintheorem} for the substitution of the variable $z$ with $\psi$-classes.  

\subsection{Equivariant case} \label{Seceq}
We now comment on how to adjust the proof of the wall-crossing formula in the non-proper equivariant setting. The issue is that $\BM\FuM^{\epsilon_0}_{n,[m]}(X)$ is not proper. To conclude that the total residue at $z=0$ vanishes in (\ref{therelation}), we have to find a space, which admits a proper $\BC^*_z$-equivariant projection from $\BM\FuM^{\epsilon_0}_{n,[m]}(X)$ and has a trivial $\BC^*_z$-action. Such space is  provided by the coarse symmetric product, and the projection is just the Hilbert--Chow map followed by the contraction map of FM degenerations,
\[\pi \colon \BM\FuM^{\epsilon_0}_{n,[m]}(X) \rightarrow X^{n}/S_n.\]

Consider the completed localised cohomology of $\BM\FuM^{\epsilon_0}_{n,[m]}(X)$, 
\[ H^*(\BM\FuM^{\epsilon_0}_{n,[m]}(X)^{\BC^*_z\times T})[\![ t_i^\pm,z^\pm]\!],\]
such that we expand all rational functions in $g_i$ and $z$ in the range $|z|> |t_i| >1$. The completion is needed, because the $T$-equivariant cohomology  is not nilpotent.  The residue at $z=0$ of an element in the group above is defined as the coefficient of the term $1/z$.  To prove that the residue relation (\ref{therelation}) holds, it is sufficient to establish 
that 
\begin{equation*}
\pi^T_*\left(\frac{[\BM\FuM^{\epsilon_0}_{n,[m]}(X)^{\BC^*_z\times T}]^{\mathrm{vir}}\cap A\cdot B}{e_{\BC^*_z\times T}(N^\mathrm{vir})}\right)
\in H^*((X^{n}/S_n)^T)[\![ t_i^\pm,z]\!].
\end{equation*}

Using commutativity of the diagram below and the properness of  projection to the symmetric product,

\begin{equation*} 
\begin{tikzcd}[row sep=small, column sep = small]
	\BM\FuM^{\epsilon_0}_{n,[m]}(X)^{\BC^*_z\times T} \arrow[d, "\pi^T"] \arrow[r] &  \BM\FuM^{\epsilon_0}_{n,[m]}(X) \arrow[d,"\pi"] & \\
	(X^{n}/S_n)^T \arrow[r,"\iota"] & X^{n}/S_n
\end{tikzcd}
\end{equation*}
we conclude that 
\[ \iota_*\pi^T_*\left(\frac{[\BM\FuM^{\epsilon_0}_{n,[m]}(X)^{\BC^*_z\times T}]^{\mathrm{vir}}\cap A \cdot B}{e_{\BC^*_z\times T}(N^\mathrm{vir})}\right) \in  H^*_T(X^{n}/S_n)\otimes_{\BC[t_i]} \BC[\![t^\pm_i, z]\!] . \]
By localisation \cite[Theorem 1]{EG}, the localised class must also be contained in $H^*((X^{n}/S_n)^T)[\![t_i^\pm, z]\!]$. 
This proves that the total residue at $z=0$ in (\ref{therelation}) is indeed zero. We then proceed as in the non-equivariant case. 

\bibliographystyle{amsalpha}
\bibliography{HilbFM}
\end{document}